\documentclass[11pt]{amsart}

\usepackage{amssymb, latexsym, amsfonts, pigpen, enumitem, mathrsfs, amsthm, bbm, stackrel, xcolor, tikz, tikz-cd, adjustbox, nicefrac, eucal}
\usepackage[matrix,arrow,curve]{xy}
\usepackage[colorlinks=true,pagebackref=true,citecolor=cyan,linkcolor=magenta]{hyperref}
\hypersetup{linktocpage}

\usepackage{verbatim}
\usepackage[letterpaper,top=0.9in, bottom=0.9in, left=0.9in, right=0.9in]{geometry}

\numberwithin{equation}{section}

\newtheorem{theorem}[equation]{Theorem}
\newtheorem{lemma}[equation]{Lemma}
\newtheorem{proposition}[equation]{Proposition}
\newtheorem{corollary}[equation]{Corollary}
\newtheorem{theoremintr}{Theorem}

\theoremstyle{definition}
\newtheorem{definition}[equation]{Definition}

\newtheorem{remark}[equation]{Remark}

\newtheorem{notation}[equation]{Notation}

\newcommand{\Witt}{\mathbf{W}}
\newcommand{\GWitt}{\mathbf{GW}}
\newcommand{\KMil}{\mathbf{K}^{\mathrm{M}}}
\newcommand{\KMiltwo}{\mathbf{k}^{\mathrm{M}}}
\newcommand{\KMW}{\mathbf{K}^{\mathrm{MW}}}
\newcommand{\WGr}{\mathrm{WGr}}
\newcommand{\hmlg}{H}
\newcommand{\cyc}{Z}
\newcommand{\MU}{\mathbf{MU}}
\newcommand{\MSU}{\mathbf{MSU}}
\newcommand{\kgl}{\mathbf{kgl}}
\newcommand{\kq}{\mathbf{kq}}
\newcommand{\EM}{\mathbf{M}}
\newcommand{\Sp}{{\mathrm{Sp}}}
\newcommand{\SL}{{\mathrm{SL}}}
\newcommand{\Gr}{{\mathrm{Gr}}}
\newcommand{\SGr}{{\mathrm{SGr}}}
\newcommand{\Proj}{\mathbb{P}}
\newcommand{\Hom}{{\mathrm{Hom}}}
\newcommand{\Ext}{{\mathrm{Ext}}}
\newcommand{\Z}{{\mathbb{Z}}}
\newcommand{\MSL}{\mathbf{MSL}}
\newcommand{\MGL}{\mathbf{MGL}}
\newcommand{\W}{\mathbf{W}}
\newcommand{\MWL}{\mathbf{MWL}}
\newcommand{\A}{\mathbb{A}}
\newcommand{\PP}{\mathbb{P}}
\newcommand{\NN}{\mathbb{N}}
\newcommand{\Th}{\mathrm{Th}}
\newcommand{\id}{\operatorname{id}}
\newcommand{\colim}{\operatorname*{colim}}
\newcommand{\slice}{\mathsf{s}}
\newcommand{\effcover}{\mathsf{f}}
\newcommand{\vs}{\widetilde{\slice}}
\newcommand{\vf}{\widetilde{\effcover}}

\newcommand{\struct}{\mathcal{O}}
\newcommand{\Gm}{{\mathbb{G}_m}}

\newcommand{\SH}{\mathbf{SH}}
\newcommand{\T}{\mathrm{T}}
\newcommand{\Spec}{\mathrm{Spec}}
\newcommand{\coker}{\mathrm{coker}}
\newcommand{\kernel}{\mathrm{ker}}
\newcommand{\sph}{\mathbbm{1}}

\newcommand{\QQ}{\mathbb{Q}}
\newcommand{\RR}{\mathbb{R}}

\newcommand{\Pic}{\operatorname{Pic}}

\newcommand{\chark}{\operatorname{char}}
\newcommand{\EE}{\mathsf{E}}
\newcommand{\etatop}{\eta_{\mathrm{top}}}

\newcommand{\SmF}{\mathbf{Sm}_F}

\newcommand{\SHF}{\mathbf{SH}(F)}
\newcommand{\determ}{\mathrm{det}}

\newcommand{\hyper}{\mathsf{h}}
\newcommand{\dd}{\mathsf{d}}
\newcommand{\Sq}{\mathsf{Sq}}
\newcommand{\pr}{\mathrm{pr}}
\newcommand{\inc}{\mathrm{inc}}

\let\directsum=\oplus
\let\iso=\cong

\begin{document}

\title{Slices of the special linear algebraic cobordism spectrum}
\author{Ahina Nandy}
\address{Radboud Universiteit, Mathematical Institute, Postbus 9010, 6500 GL Nijmegen, Netherlands}
\email{\href{mailto:ahina.nandy@ru.nl}{ahina.nandy@ru.nl}}
\author{Oliver R\"ondigs}
\address{Universit\"at Osnabr\"uck, Institut f\"ur Mathematik, 49069 Osnabr\"uck, Germany}
\email{\href{mailto:oliver.roendigs@uni-osnabrueck.de}{oliver.roendigs@uni-osnabrueck.de}}
\author{Egor Zolotarev}
\address{LMU M\"unchen, Mathematisches Institut, Theresienstr. 39, 80333 M\"unchen, Germany}
\email{\href{mailto:zolotarev@math.lmu.de}{zolotarev@math.lmu.de}, \href{mailto:zolotarev-egv@yandex.ru}{zolotarev-egv@yandex.ru}}
\urladdr{\href{https://sites.google.com/view/egor-zolotarev}{https://sites.google.com/view/egor-zolotarev}}

\date{}

\thanks{}

\begin{abstract}
Let $F$ be a field of exponential characteristic $e$.
We compute the slices of $\MSL[\nicefrac{1}{e}]$, where $\MSL$ is the special linear algebraic cobordism spectrum defined by Panin and Walter. 
The answer is expressed in terms of the second page of the Adams--Novikov spectral sequence for the special unitary cobordism spectrum, which was explicitly determined by Novikov. 
Its applicability is demonstrated by computations with the slice spectral sequence for $\MSL$, which determine the first few Milnor--Witt stems of its homotopy groups (up to the third) in terms of very effective hermitian $K$-theory. 
We also establish a decomposition of the rational special linear algebraic cobordism spectrum over an arbitrary qcqs scheme.
\end{abstract}
\maketitle
\tableofcontents
\section{Introduction}
The slice filtration (also known as the motivic Postnikov tower) on the Morel--Voevodsky $\PP^1$-stable $\A^1$-homotopy category was introduced by Vladimir Voevodsky \cite{Voev_slices} to give a conceptual construction of the motivic spectral sequence for higher algebraic $K$-theory.
Its existence was conjectured by Alexander Beilinson and Stephen Lichtenbaum in analogy with the classical Atiyah--Hirzebruch spectral sequence in topology. 
This approach was successfully realized over fields by Vladimir Voevodsky \cite{Voev_zero_slice} and Marc Levine \cite{Levine_coniveau} (see also a recent paper of Bachmann--Elmanto--Morrow over arbitrary qcqs schemes \cite{BEM_AHSS}). 

Since then, the slice filtration has become a powerful computational tool in motivic homotopy theory. 
For example, this filtration was used in a joint work of the second author with Paul Arne \O stv\ae r for a new proof of the Milnor conjecture on quadratic forms \cite{RO16} (the original proof was given by Orlov--Vishik--Voevodsky \cite{OVV_Milnor_qf}). 
Furthermore, the slice filtration was applied to a systematic study of the stable homotopy groups of motivic spheres in joint works of the second author with Markus Spitzweck and Paul Arne \O stv\ae r \cite{RSO19, RSO24}. 
This approach is based on computations with the spectral sequence associated with the slice filtration for the motivic sphere spectrum $\sph$. 
The $E^1$-page of this spectral sequence is given by motivic cohomology with coefficients in the $E_2$-page of the topological Adams--Novikov spectral sequence for the topological sphere spectrum. 
It can be computed in a certain range using topological computations together with an analysis of certain motivic cohomology groups, mainly based on Voevodsky's solution of the Bloch--Kato conjecture \cite{Voevodsky_MC, Voevodsky_BK}. 
Moreover, the $\dd_1$-differentials in this spectral sequence (we call them first slice differentials) are explicitly describable with help of the motivic Steenrod algebra that was determined for primes different from the characteristic in \cite{Voevodsky_Steenrod} and \cite{HKO_Steenrod}.

In this paper we study the slice filtration for the special linear algebraic cobordism spectrum $\MSL$. 
This spectrum was introduced by Ivan Panin and Charles Walter in \cite{PW-Thom} as a version of ``oriented algebraic cobordism''\footnote{Note that there is another spectrum that might also be viewed as an oriented algebraic cobordism spectrum. It is called the metalinear algebraic cobordism spectrum and admits a more common version of ``quadratic'' orientation.}.
This spectrum is extensively studied now (see e.g., \cite{BH_eta, LYZ, ZolMSL}) and its first few slices were computed in work of the first author \cite{NanMSL}. 
In this paper we compute them completely. 
The answer is given away from the characteristic of the chosen base field, since the proof uses the slices of algebraic cobordism $\MGL$.
These were computed by Spitzweck \cite{SpiMGL1} assuming the Hopkins--Morel--Hoyois equivalence \cite{HoyMGL}, which is currently known to hold away from the characteristic. 
The following result is a straightforward combination of Theorems \ref{thm_conner_floyd} and \ref{slices_MSL}.
\begin{theoremintr}\label{thmintro.a}
    Suppose that $F$ is a field of characteristic zero. The slices of $\MSL$ are given by
    \[ \slice_q(\MSL)\cong \bigoplus_{0\leq n<\nicefrac{q}{2}}\Sigma^{q+2n,q}\EM\Z/2^{p(\lfloor \nicefrac{n}{2}\rfloor)}\oplus 
    \Sigma^{2q,q}\EM\Z^{p(q)-p(q-1)}, \]
    where $p(-)$ is the partition function and $\lfloor-\rfloor$ is the floor function. 
    If $F$ is a field of characteristic $p>0$, then the same result holds after inverting $p$.
\end{theoremintr}

The above answer has an interpretation in terms of the $E_2$-page of the Adams--Novikov spectral sequence for the (topological) special unitary cobordism spectrum, see Remark \ref{remark_ANSS}. 
The proof of this theorem is based on the work of the third author \cite{ZolMSL} and proceeds via the computation of the slices of the $c_1$-spherical cobordism spectrum $\MWL$, which is equivalent to the $\eta$-cofiber of $\MSL$. 
We also obtain a comparison with the slices of the very effective hermitian $K$-theory via the morphism of motivic $\mathbb{E}_\infty$-ring spectra $\MSL\to \kq$ in Proposition \ref{prop:slices-msl-kq}. 
This comparison is used multiple times in further slice computations.

After giving a proof of Theorem~\ref{thmintro.a}, we perform explicit computations with the slice spectral sequence. 
More precisely, in Section \ref{section:4}, we compute the first slice differentials for $\MGL$ and then deduce formulas for the first slice differentials for $\MSL$ in a certain range via $\MWL$. 
These computations lead to a determination of the first four nontrivial columns of the $E^2$-page of the slice spectral sequence for $\MSL$. 
Section \ref{section:5} provides several higher slice differentials in order to push the previous computations to the $E^\infty$-page. 
Along the way, a few higher slice differentials for the motivic sphere spectrum $\sph$ are identified, which may be of independent interest. 
Section \ref{section:6} applies the obtained results to compute the homotopy groups of $\MSL$ up to the third Milnor--Witt degree, using the $\eta$-arithmetic square. 
This leads to the following result; for notation see the subsection below.

\begin{theoremintr}
    Suppose that $F$ is a field of characteristic zero and $n$ is an integer. 
    Then the canonical morphism $\MSL\to\kq$ induces
    \begin{enumerate}
        \item isomorphisms $\pi_{n+m,n}(\MSL)\cong\pi_{n+m,n}(\kq)$ for $m\leq 2$,
        \item an epimorphism $\pi_{n+3,n}(\MSL)\twoheadrightarrow\pi_{n+3,n}(\kq)$ with kernel $\KMil_{3-n}(F)/(\partial^2_\infty\Sq^6\pr^\infty_2H^{-n-4,-n})$.
    \end{enumerate}
    If $F$ is a field of characteristic $p>0$, then the same result holds after inverting $p$.
\end{theoremintr}

In the last section, we obtain a decomposition of the rational special linear algebraic cobordism spectrum over an arbitrary qcqs scheme $S$ into the product of polynomial homotopy commutative ring spectra over the rational motivic cohomology spectrum $\EM\QQ$ and rational Witt motivic cohomology spectrum $\W\QQ$ (see Theorem \ref{thm:rational_decomp}):
    \[ \EM\QQ[y_2,y_3,\dots]\times \W\QQ[u_4,u_8,\dots]\cong \MSL\otimes\QQ. \]
The polynomial generators here arise from the polynomial generators of the geometric diagonal of $\MSL\otimes\QQ$ over $\Z$, which is computed via Betti realizations, see Lemma \ref{lem:geom_diag_overZ}. 
The proof for the plus part is essentially based on the degeneration of the slice spectral sequence for $\MSL\otimes\QQ$ over fields, while the proof for the minus part uses real Betti realization and the work of Tom Bachmann \cite{Bachmann_real_etale}.

\subsection*{Notation and terminology} 
Throughout the text, $F$ is a base field and $\Lambda=\Z[\nicefrac{1}{e}]$, where $e$ is the exponential characteristic of $F$. 
Starting from Section \ref{section:4}, we assume that the characteristic of $F$ is different from $2$.
The case $\chark(F)=2$ is treated separately in Appendix \ref{appendix_char2}. 
Let $\SmF$ denote the category of smooth $F$-schemes and $\SHF$ the $\PP^1$-stable $\A^1$-homotopy category over $F$.
We also denote by $\Sigma^{s+(w)}=\Sigma^{s+w,w}$ the suspension with the motivic sphere $S^{s+(w)}=S^{s+w,w}=S^s\wedge\mathbb{G}_m^w$, and by $\pi_{s+(w)}\mathsf{E}=\pi_{s+w,w}\mathsf{E}$ the corresponding homotopy group of a spectrum $\mathsf{E}\in\SHF$. 
In addition, we use the abbreviation $\pi_s\mathsf{E}$ for the direct sum
\[ \pi_{s+(\star)}\mathsf{E}=\bigoplus_{w\in \Z}\pi_{s+(w)}\mathsf{E} \]
considered as a module over the Milnor--Witt $K$-theory by Morel's theorem $\pi_{0}\sph\cong \KMW_{-\star}(F)$ \cite[Theorem 1.23]{Morel_book}. 
The notation $\eta\in\pi_{0+(1)}\sph$ stands for the motivic Hopf element, $\hyper\in\pi_{0+(0)}\sph\cong\mathbf{GW}(F)$ for the hyperbolic plane $2+\eta[-1]$, and $\rho\in\pi_{0-(1)}\sph$ for the map $\sph\to \Gm$ corresponding to $-1\in F^\times$.
By abuse of notation, the same symbol $\rho$ also denotes the corresponding element in integral and mod-$2$ motivic cohomology.

We abbreviate by $H^{s,w}$ (resp. $h^{s,w}$, $h^{s,w}_n$) the integral (resp. mod-$2$, mod-$n$) motivic cohomology group $H^s(F,\Z(w))$ (resp. $H^s(F,\Z/2(w))$, $H^s(F,\Z/n(w))$). 
The notation $\partial^a_b$ is frequently used for the connecting homomorphism of the form $h^{s,w}_a\to h^{s+1,w}_b$. 
In turn, $\pr^a_b$ and $\inc^a_b$ denote the projection and inclusion homomorphisms $h^{s,w}_a\to h^{s,w}_b$. 
Here $a$ and $b$ can take the value $\infty$, and in this case $h^{s,w}_\infty=H^{s,w}$. 
Let $\tau$ denote the generator of $h^{0,1}$, as usual. 
We also use notation $\Sq^i$ for the $i$-th motivic Steenrod square operation.
The only other Steenrod power that will appear is $\mathsf{P}^1$ at the prime $3$.

For a motivic spectrum $\EE\in\SHF$, the notation $\EE_\Lambda$ stands for its $\Lambda$-localization $\EE\otimes\Lambda\in \SH(F)\otimes\Lambda$. 
For the convenience of the reader, here is the list of the main motivic and topological spectra used in the text:

\begin{tabular}{p{3.5cm}|l}
    $\sph$ & motivic sphere spectrum \\
    $\EM A$ & motivic Eilenberg--MacLane spectrum of $A$ \\
    $\MSL$ & special linear algebraic cobordism \cite[\S 4]{PW-Thom} \\
    $\MGL$, $\MWL$ & algebraic cobordism and $c_1$-spherical algebraic cobordism \cite[\S 2]{ZolMSL} \\
    $\mathbf{KGL}$, $\mathbf{KQ}$ & algebraic and hermitian $K$-theory \\
    $\kgl$, $\kq$ & (very) effective algebraic and very effective hermitian $K$-theory \\
    $\mathbf{kw}$ & connective Witt theory \\
    $\MSU$ & special unitary cobordism \\
    $\MU$, $\W$ & complex cobordism and $c_1$-spherical cobordism \cite{Conner_Floyd} 
\end{tabular}

\subsection*{Acknowledgements}

Oliver R\"ondigs acknowledges support by a DFG research grant with project number 539085450. 
Egor Zolotarev is supported by the DFG research grant AN 1545/4-1 and thanks Alexey Ananyevskiy for useful discussions.

\section{Slice filtration and cobordism spectra}
This section constitutes an overview of known facts that are essential to the main arguments. 
More precisely, we recall the slice filtration, and the constructions and basic properties of the algebraic cobordism spectra $\MSL$ \cite[\S 4]{PW-Thom} and $\MWL$ \cite[\S 2]{ZolMSL}. 
We also discuss the action of certain cohomological operations on the Lazard ring, which appear in \cite[\S 4]{ZolMSL}.
\subsection{Voevodsky's slice filtration}
Denote by $i_q\colon\Sigma^{q+(q)}\SH^\mathrm{eff}(F)\to \SHF$ the inclusion of the localizing subcategory generated by $\Sigma^{q+(q)}$-suspensions of smooth $F$-schemes. 
This functor admits a right adjoint $r_q$. 
Following Voevodsky \cite{Voev_slices0}, the functor $\effcover_q:=i_q\circ\,r_q$ is called the \textit{$q$-th effective cover}. 
Any spectrum $\EE\in\SHF$ is supplied with an exhaustive filtration
\[ \cdots\to\effcover_{n+1}(\EE)\to\effcover_n(\EE)\to\effcover_{n-1}(\EE)\to\cdots\to\EE.\]
The \textit{$q$-th slice of a spectrum $\EE\in\SHF$}, denoted by $\slice_q(\EE)$, is the $q$-th successive quotient of this filtration. 
There is a natural cofiber sequence in $\SHF$ 
\begin{equation}\label{equat_def_slices}
    \effcover_{q+1}(\EE)\to\effcover_q(\EE)\to\slice_q(\EE)\to\Sigma\effcover_{q+1}(\EE).
\end{equation}
For formal reasons, the slices of a spectrum $\EE\in\SHF$ are (strict) modules over the zeroth slice of the sphere spectrum $\slice_0(\sph)$, which is equivalent to the motivic Eilenberg--MacLane spectrum $\EM\Z$; see \cite{Levine_coniveau, Voev_zero_slice}. 
Thus if $\EE$ is a $\Lambda$-local motivic spectrum $\EE\in\SH(F)\otimes\Lambda$, then the slices $\slice_q(\EE)$ are (strict) modules over $\EM\Lambda$.

Let $\EE$ be a motivic spectrum over $F$. 
Then the slice filtration on $\EE$ gives a spectral sequence of the form 
\begin{equation}
    E^1_{n+m,q,n}(\EE)=\pi_{m+(n)}\slice_q(\EE)\leadsto \pi_{m+(n)}(\EE).
\end{equation}
The last arrow means that the $E^\infty$-page is somehow related to $\pi_{*+(\star)}(\EE)$, but we do not claim any convergence property of the above spectral sequence in general. 
The $\dd_1$-differential of this spectral sequence is obtained by applying $\pi_{m+(n)}(-)$ to the following composition
\begin{equation}
    \slice_q(\EE)\to\Sigma\effcover_{q+1}(\EE)\to\Sigma\slice_{q+1}(\EE),
\end{equation}
where both morphisms are defined via the cofiber sequence \eqref{equat_def_slices}. The $r$-th slice differential has the following form:
\[ \dd_r^\EE(q)\colon E^r_{n+m,q,n}(\EE)\to E^r_{n+m-1,q+r,n}(\EE), \]
and it is sometimes denoted simply by $\dd_r(q)$ or even $\dd_r$. 
Taking the sum of these differentials over $n\in \Z$ produces a morphism of modules over Milnor $K$-theory $\KMil_\star(F)$ by \cite[Lemma 4.14]{RSO19}.
\subsection{Algebraic cobordism spectra}
We begin with the recollection on the Panin--Walter special linear algebraic cobordism spectrum $\MSL$. 
Consider the Grassmannian scheme $\Gr_n(\A^m)$ of $n$-dimensional linear subspaces of $\A^{m}$ and denote by $\gamma_{n,m}$ the tautological vector bundle over $\Gr_n(\A^m)$.
\begin{definition}
     The \textit{special linear Grassmannian} $\SGr_n(\A^m)$ is the smooth $F$-scheme given by the complement of the zero section of the determinant bundle over $\Gr_n(\A^m)$ 
    \[\SGr_n(\A^m):=\determ(\gamma_{n,m})^\circ\in\SmF.\]
    Let $\gamma_{n,m}^{\SL}$ denote the pullback of the tautological vector bundle along $\SGr_n(\A^m)\to\Gr_n(\A^m)$.
\end{definition}
\begin{remark}
    The colimit of these Grassmannians along the embeddings $\SGr_n(\A^m)\hookrightarrow \SGr_n(\A^{m+1})$ is motivically equivalent to the classifying space of the special linear group (see e.g., the proof of \cite[Theorem 4.1.1]{AHW_BG}):
    \[\SGr_n:=\colim_m\SGr_n(\A^m)\cong_{\mathrm{mot}} \mathrm{BSL}_n.\]
\end{remark}
Let $\Th(\gamma_n^\SL)$ be the colimit of the following tower of morphisms of presheaves
 \[\cdots\to \Th(\gamma_{n,m-1}^\SL)\to \Th(\gamma_{n,m}^\SL)\to \Th(\gamma_{n,m+1}^\SL)\to\cdots. \]
\begin{definition}[{\cite[Definition 4.2]{PW-Thom}}]
The \textit{special linear algebraic cobordism spectrum} $\MSL$ is the motivic spectrum with $n$-th space given by $\Th(\gamma_n^\SL)$, and its structure maps are induced by $\SGr_n\to \SGr_{n+1}$. 
\end{definition}
The above spectrum can be endowed with an $\mathbb{E}_\infty$-ring structure, either by constructing the corresponding symmetric $\T^2$-spectrum \cite[\S 4.2(a)]{PW-Thom} or by using the Bachmann--Hoyois motivic Thom functor (see \cite[Example 16.22]{BH-Norms}). 
Moreover, by construction there is a forgetful ring morphism to Voevodsky's algebraic cobordism spectrum 
\begin{equation}\label{MSL_to_MGL}
    \MSL\to\MGL.
\end{equation}

We turn to the $c_1$-spherical algebraic cobordism spectrum $\MWL$. 
It is also constructed by using Thom spaces over appropriate Grassmannians. 
\begin{definition}
    The \textit{Wall Grassmannian} $\WGr_n(\A^m)$ is a smooth $F$-scheme given by the complement of the zero section of the line bundle $\determ(\gamma_{n,m})\boxtimes \struct(1)$ over $\Gr_n(\A^m)\times \Proj^1$
    \[\WGr_n(\A^m):=(\determ(\gamma_{n,m})\boxtimes \struct(1))^\circ\in\SmF. \]
    We denote by $\gamma^{\mathrm{W}}_{n,m}$ the pullback of $\gamma_{n,m}$ along the projection $\WGr_n(\A^m)\to\Gr_n(\A^m)$.
\end{definition}
Now the construction of the respective Thom spectrum is similar to the construction of $\MSL$. Let $\Th(\gamma_n^{\mathrm{W}})$ be the colimit of the sequence 
\[\cdots\to\Th(\gamma^{\mathrm{W}}_{n,m-1})\to\Th(\gamma^{\mathrm{W}}_{n,m})\to\Th(\gamma^{\mathrm{W}}_{n,m+1})\to\cdots.\]
\begin{definition}[{\cite[Definition 2.5 and Proposition 2.11]{ZolMSL}}] 
    The \textit{$c_1$-spherical algebraic cobordism spectrum} $\MWL$ is the motivic spectrum with $n$-th space given by $\Th(\gamma_n^{\mathrm{W}})$ and its structure maps are induced by $\WGr_n\to \WGr_{n+1}$.
\end{definition}
The respective spectrum can be endowed with a natural $\MSL$-module structure; see \cite[Proposition 2.6]{ZolMSL}. 
Moreover, the ring morphism \eqref{MSL_to_MGL} factors though $\MWL$ such that both maps $\mathrm{forg}\colon\MSL\to\MWL$ and $\mathrm{forg}'\colon\MWL\to\MGL$ are morphisms of $\MSL$-modules. These arrows are included in the following cofiber sequences.
\begin{theorem}[see {\cite[Theorem A]{ZolMSL}}]\label{MWL_cofib} 
    There are cofiber sequences in $\SHF$ \begin{enumerate} \item $\Sigma^{(1)}\MSL\xrightarrow{\eta}\MSL\xrightarrow{\mathrm{forg}}\MWL\xrightarrow{d} \Sigma^{1+(1)}\MSL$, \item $\MWL\xrightarrow{\mathrm{forg}'}\MGL\xrightarrow{\Delta}\Sigma^{2+(2)}\MGL\to \Sigma\MWL$, \end{enumerate} where $\eta\in\pi_{0+(1)}\sph$ is the first motivic Hopf element and $\Delta$ is the cohomological operation that corresponds to $c_1(\determ\,\gamma)\cdot c_1(\determ\,\gamma^\vee)$ under the Thom isomorphism $\MGL^{4,2}(\MGL)\cong \MGL^{4,2}(\mathrm{BGL}).$ %\ahina{degree shift}
\end{theorem}
\begin{notation}
    In what follows, $d:\MWL\to\Sigma^{1+(1)}\MSL$ denotes boundary morphism of the first cofiber sequence, and $\delta:\MWL\to\Sigma^{1+(1)}\MWL$ denotes the composition of $d$ with $\Sigma^{1+(1)}\mathrm{forg}$.
    Sometimes the letter $d$ is used in other contexts, such as differentials in spectral sequences; we hope the meaning will always be implicitly clear.
\end{notation}
\subsection{Operations on the Lazard ring and Conner--Floyd homology}
Consider the coefficient ring of the complex cobordism spectrum $\MU_*$. 
This ring can be identified with the Lazard ring by Quillen's theorem \cite{quillen.mu}. 
We denote by \[\Delta_{n}\colon\MU_{n}\to\MU_{n-4}\] the action of the cohomological operation $c_1(\determ\,\mathrm{EU})\cdot c_1(\determ\,\mathrm{EU}^\vee)\in\MU^4(\mathrm{BU})\cong\MU^4(\MU)$ on the coefficient ring. %\ahina{similar degree shift}
\begin{notation}
    The groups $\W_{n}:=\kernel(\Delta_n)$ are free abelian and concentrated in even non-negative degrees.  
\end{notation}
Similarly, consider the action of the $\MU$-cohomological operation $-c_1(\det\,\mathrm{EU}^\vee)\in\MU^2(\mathrm{BU})\cong\MU^2(\MU)$ on the coefficient ring $\MU_n\to \MU_{n-2}$. 
This action lifts to the correctly defined homomorphisms of abelian groups $\delta_n\colon\W_{n}\to\W_{n-2}$ (see \cite[\S 6]{CLP}).
Since $\delta_{n-2}\circ\delta_n=0$, there results a chain complex of abelian groups \[\cdots\xrightarrow{\delta_{n+4}}\W_{n+2}\xrightarrow{\delta_{n+2}}\W_{n}\xrightarrow{\delta_{n}}\W_{n-2}\xrightarrow{\delta_{n-2}}\cdots.\]
\begin{notation}
    Let $\cyc_{n}(\W)$ (resp. $B_{n}(\W)$, resp. $\hmlg_{n}(\W)$) denote the group of cycles (resp. boundaries, resp. homology) of the chain complex $(\W_{*},\delta_{*})$ in the spot corresponding to $\W_{n}$. 
    These groups are concentrated in even nonnegative degrees.
\end{notation}
The following classical computation is due to Conner and Floyd \cite{Conner_Floyd} (see also \cite[Appendix~B]{ZolMSL} for an overview).
\begin{theorem}\label{thm_conner_floyd}
    Let $n$ be a nonnegative integer.
    \begin{enumerate}
        \item The group $\cyc_{2n}(\W)$ is a free abelian group of rank $p(n)-p(n-1)$, where $p(-)$ is the partition function.
        \item The group $\hmlg_{2n}(\W)$ is trivial if $n$ is odd and isomorphic to $(\Z/2)^{p(k)}$ for $n=4k$ or $n=4k+2$. 
    \end{enumerate}
\end{theorem}
\begin{remark}
    The above notation agrees with that of the previous subsection. 
    More precisely, $\pi_{\star+(\star)}(\MGL_\Lambda)$ is isomorphic to $\Lambda\otimes\MU_{2\star}$ and the action of $\Delta\colon\MGL\to\Sigma^{2+(2)}\MGL$ on the $\Lambda$-local geometric diagonal coincides with \[\Lambda\otimes\Delta_{2\star}\colon\Lambda\otimes\MU_{2\star}\to\Lambda\otimes\MU_{2\star-4}\] 
    by \cite[Lemma 4.5]{ZolMSL}. 
    Similarly, $\pi_{\star+(\star)}(\MWL_\Lambda)$ is isomorphic to $\Lambda\otimes\W_{2\star}$ and the action of $\delta\colon\MWL\to\Sigma^{1+(1)}\MWL$ coincides (on the $\Lambda$-local geometric diagonal) with the map  \[\Lambda\otimes\delta_{2\star}\colon\Lambda\otimes\W_{2\star}\to\Lambda\otimes\W_{2\star-2}\]
    by \cite[Proposition 4.10, Lemma 4.13]{ZolMSL}.
\end{remark}
\section{Slices of some cobordism spectra}
In this section, we compute the slice filtrations of the $c_1$-spherical cobordism spectrum $\MWL_\Lambda$, the special linear algebraic cobordism spectrum $\MSL_\Lambda$, and its $\eta$-periodic version $\MSL_\Lambda[\eta^{-1}]$. 
In addition, we compare their slices with the slices of the respective ``connective'' $K$-theory spectra (effective algebraic $K$-theory $\kgl_\Lambda$, resp. very effective hermitian $K$-theory $\kq_\Lambda$, resp. connective Witt theory $\mathbf{kw}_\Lambda$). 
\subsection{The \texorpdfstring{$c_1$}{c1}-spherical algebraic cobordism spectrum}
Recall the following standard observation about the motivic Eilenberg--MacLane spectra.
\begin{lemma}\label{lemma_EM_ses}
    A short exact sequence $0\to A\xrightarrow{\varphi} B\xrightarrow{\psi} C\to 0$ of abelian groups 
    induces a cofiber sequence  \[\EM\Lambda\otimes A\xrightarrow{\EM\Lambda\otimes\varphi}\EM\Lambda\otimes B\xrightarrow{\EM\Lambda\otimes\psi}\EM\Lambda\otimes C\to \Sigma\EM\Lambda\otimes A\]
    of $\EM\Lambda$-modules.
\end{lemma}
\begin{proof}
    This follows from \cite[Proposition 4.13.(2)]{HoyMGL} (see also \cite[Lemma A.3]{RSO19}).
\end{proof}
The next input essential for our computation is the partial solution of the Voevodsky conjecture about the slices of the algebraic cobordism spectrum, due to Spitzweck \cite[Theorem 4.7]{SpiMGL1}, \cite[Theorem 3.1]{SpiMGL2} and Hopkins--Morel--Hoyois \cite[Theorem 7.12]{HoyMGL}. It states that there is an isomorphism of $\EM\Lambda$-modules:
\begin{equation}\label{slices_MGL}
       \slice_q(\MGL_\Lambda)\cong \Sigma^{q+(q)}\EM\Lambda\otimes\MU_{2q}.
\end{equation}
\begin{theorem}\label{slices_MWL}
    The slices of $\MWL_\Lambda$ are given by  \[\slice_q(\MWL_\Lambda)\cong \Sigma^{q+(q)}\EM\Lambda\otimes\W_{2q}\] 
    as $\EM\Lambda$-modules. 
    More precisely, the canonical morphism $\slice_q(\MWL_\Lambda)\to\slice_q(\MGL_\Lambda)$ coincides with $\Sigma^{q+(q)}\EM\Lambda\otimes\mathrm{inc}_{2q}$ under the above identifications. 
    Here $\mathrm{inc}_{2q}$ is the inclusion $\mathrm{inc}_{2q}\colon\W_{2q}\hookrightarrow\MU_{2q}$.
\end{theorem}
\begin{proof}
    Applying the $q$-th slice functor to the $\Lambda$-local version of the cofiber sequence \ref{MWL_cofib}(2) produces the following cofiber sequence of $\EM\Lambda$-modules: \[\slice_q(\MWL_\Lambda)\to \slice_q(\MGL_\Lambda)\xrightarrow{\slice_q(\Delta)} \slice_q(\Sigma^{2+(2)}\MGL_\Lambda)\cong\Sigma^{2+(2)}\slice_{q-2}(\MGL_\Lambda).\] 
    Using the equivalence \eqref{slices_MGL} and \cite[Lemma 4.5]{ZolMSL}, the last morphism can be identified with \[\Sigma^{q+(q)}\EM\Lambda\otimes\Delta_{2q}\colon\Sigma^{q+(q)}\EM\Lambda\otimes\MU_{2q}\to\Sigma^{q+(q)}\EM\Lambda\otimes\MU_{2q-4}.\] 
    The homomorphism $\Delta_{2q}\colon\MU_{2q}\to\MU_{2q-4}$ is surjective with kernel $\W_{2q}$ by \cite[Proposition B.6]{ZolMSL}. 
    Hence, the result follows from Lemma~\ref{lemma_EM_ses}.
\end{proof}
Our next goal is to compare the slices of the $c_1$-spherical cobordism spectrum and the (very) effective cover of the spectrum $\mathbf{KGL}$ representing algebraic $K$-theory. 
Consider the Wood cofiber sequence \[\Sigma^{(1)}\mathbf{KQ}\xrightarrow{\eta}\mathbf{KQ}\xrightarrow{\mathrm{Forg}}\mathbf{KGL}\xrightarrow{\mathrm{Hyp}}\Sigma^{1+(1)}\mathbf{KQ}\]
given in \cite[Theorem 3.4]{RO16}, and \cite[Theorem 2.1(2)]{KRO_hermitian_k_theory_char2} in characteristic $2$. 
Here $\mathrm{Hyp}$ is the composition of a suitable suspension of the hyperbolic map with the Bott periodicity equivalence of $\mathbf{KGL}$.
The usual $\SL$-orientation of the hermitian $K$-theory spectrum $\mathbf{KQ}$ corresponds to a map $\MSL\to\mathbf{KQ}$ (in fact there is an $\mathbb{E}_\infty$-ring refinement of this morphism; see \cite[Remark 7.10 and Remark 7.3]{KQ_Gorenstein}). 
It induces a morphism 
\begin{center}
\begin{tikzcd}
\Sigma^{(1)}\MSL \ar[r,"\eta"] \ar[d] & \MSL \ar[r] \ar[d] & \MWL \ar[d,dotted] \ar[r] &\Sigma^{1+(1)}\MSL \ar[d]\\ \Sigma^{(1)}\mathbf{KQ} \ar[r,"\eta"] & \mathbf{KQ} \ar[r] & \mathbf{KGL}  \ar[r] & \Sigma^{1+(1)}\mathbf{KQ}
\end{tikzcd}
\end{center}
of cofiber sequences.
Since $\MWL$ is effective, any resulting map $\MWL\to \mathbf{KGL}$ factors uniquely through $\kgl:=\effcover_0(\mathbf{KGL})$.

By the works of Levine \cite{Levine_coniveau} and Voevodsky \cite{Voev_slices} the slices of $\kgl$ are given by \[\slice_q(\kgl)\cong\Sigma^{q+(q)}\EM\Z\ \text{for $q\geq0$}.\]
Moreover, the usual orientation map $\MGL_\Lambda\to\kgl_\Lambda$ induces the following map on the $q$-th slices (under the above identifications) \[\Sigma^{q+(q)}\EM\Lambda\otimes \mathrm{FGL}(m)_{2q}\colon\Sigma^{q+(q)}\EM\Lambda\otimes\MU_{2q}\to\Sigma^{q+(q)}\EM\Lambda;\] here $\mathrm{FGL}(m)_{2q}$ is the degree $2q$ part of the homomorphism of graded rings $\mathrm{FGL}(m)\colon\MU_*\to\Z[\beta]$ that classifies the multiplicative formal group law over $\Z[\beta]$ with $\mathrm{deg}(\beta)=2$.
\begin{lemma}\label{lemma_map_s_mwl_to_kgl}
    The map $\MWL_\Lambda\to \kgl_\Lambda$ induces the following map on the $q$-th slices 
    \[\Sigma^{q+(q)}\EM\Lambda\otimes f_{2q}\colon\Sigma^{q+(q)}\EM\Lambda\otimes\W_{2q}\to\Sigma^{q+(q)}\EM\Lambda;\] 
    here the homomorphism $f_{2q}\colon\W_{2q}\to \Z$ is the composition of the inclusion $\mathrm{inc}_{2q}\colon\W_{2q}\hookrightarrow\MU_{2q}$ and $\mathrm{FGL}(m)_{2q}$.
\end{lemma}
\begin{proof}
    By the construction of the map $\MWL_\Lambda\to\kgl_\Lambda$ the diagram 
    \begin{center}
    \begin{tikzcd}
    \MSL_\Lambda \ar[r] \ar[rd] & \MWL_\Lambda \ar[r] \ar[d] & \MGL_\Lambda \ar[ld] \\ & \kgl_\Lambda & 
    \end{tikzcd}
    \end{center}
    commutes, where the right diagonal arrow is induced by the orientation of $\kgl_\Lambda$ and the left diagonal map refines this orientation to an $\SL$-orientation. 
    Applying $\slice_q(-)$ to the right commutative triangle and using Theorem \ref{slices_MWL} with the discussion above, the result follows.
\end{proof}
\subsection{Passage to the special linear cobordism spectrum}
For the computation of the slices of the special linear algebraic cobordism spectrum, a few known lemmas about maps between the motivic Eilenberg--MacLane spectra are useful. 
Let $\EM\Lambda\text{-}\mathbf{mod}$ denote the category of (strict) $\EM\Lambda$-modules associated with the ring $\Lambda$.
\begin{lemma}\label{lemma_hom_(s,0)}
    Suppose that $A,B$ are $\Lambda$-modules. 
    Then there are natural isomorphisms  
    \[ [\EM A,\Sigma^{p}\EM B]_{\EM\Lambda\text{-}\mathbf{mod}} \cong 
    \begin{cases} 
        \Hom_\Lambda(A,B), & p=0, \\
        \Ext_\Lambda(A,B), & p=1, \\
        0, & \mathrm{otherwise}.
    \end{cases} \]
\end{lemma}
\begin{proof}
    \cite[Lemma A.3]{RSO19}.
\end{proof}
\begin{lemma}\label{lemma_EM_map}
     Let $\varphi:A\to B$ be a homomorphism of abelian groups.
     Then there exist (not necessarily unique) morphisms $\partial\colon\Sigma^{-1}\EM\Lambda\otimes\coker(\varphi)\to\EM\Lambda\otimes A$ and $\theta\colon \EM\Lambda\otimes B\to \Sigma \EM\Lambda\otimes \ker(\varphi)$ such that the sequence \[\EM\Lambda\otimes A\xrightarrow{\EM\Lambda\otimes\varphi}\EM\Lambda\otimes B
     \xrightarrow{(\theta\ \EM\Lambda \otimes \mathrm{pr})^{\mathrm{tr}}}\Sigma \EM\Lambda\otimes\kernel(\varphi)\oplus \EM\Lambda\otimes\coker(\varphi)\xrightarrow{(\Sigma\EM\Lambda\otimes\mathrm{inc}\ \Sigma\partial)} \Sigma \EM\Lambda\otimes A\] is a cofiber sequence of $\EM\Lambda$-modules. 
     Here $\mathrm{inc}\colon \kernel(\varphi)\hookrightarrow A$ and $\mathrm{pr}\colon B\to \coker(\varphi)$ are the canonical homomorphisms.
\end{lemma}
\begin{proof}
    The morphism $\EM\Lambda\otimes\varphi$ factors as the composition $\EM\Lambda\otimes A\to \EM\Lambda\otimes\mathrm{im}(\varphi)\to\EM\Lambda\otimes B$. 
    The cofiber sequence of the respective fibers in the category of $\EM\Lambda$-modules can be identified with \[\EM\Lambda\otimes\kernel(\varphi)\to \mathrm{fib}(\EM\Lambda\otimes\varphi)\to \Sigma^{-1}\EM\Lambda\otimes\coker(\varphi)\] by Lemma \ref{lemma_EM_ses}. 
    After suspension, the boundary morphism $\Sigma^{-1}\EM\Lambda\otimes\coker(\varphi)\to\Sigma^{1}\EM\Lambda\otimes\kernel(\varphi)$ belongs to the group \[[\EM\Lambda\otimes\coker(\varphi),\Sigma^{2}\EM\Lambda\otimes\kernel(\varphi)]_{\EM\Lambda\text{-}\mathbf{mod}},\] which is trivial by Lemma \ref{lemma_hom_(s,0)}. 
    Hence, the above cofiber sequence splits and a choice of a splitting provides the required maps.
\end{proof}
\begin{remark}\label{rem:splitting_non_unique}
    As written in the statement of the lemma, in general $\partial$ is not unique. 
    The reason is that it depends on the choice of a section to $\mathrm{fib}(\EM\Lambda\otimes\varphi)\to \Sigma^{-1}\EM\Lambda\otimes\coker(\varphi)$. 
    In this section, we ignore this issue since for the next computation only its existence is required. However, for the computation of the first slice differentials, we will make one specific choice, see Remark \ref{rem:choice_s3}.
\end{remark}
\begin{notation}\label{not:iota(delta)}
    Denote by $\iota_{2q-2}^*(\delta_{2q})$ the homomorphism $\W_{2q}\to\cyc_{2q-2}(\W)$ given by the pullback of the map $\delta_{2q}\colon\W_{2q}\to\W_{2q-2}$ along the inclusion $\iota_{2q-2}\colon\cyc_{2q-2}(\W)\hookrightarrow\W_{2q-2}$. 
    This just restricts the target of $\delta_{2q}$, since the image of $\delta_{2q}$ is contained in the group of cycles $\cyc_{2q-2}(\W)$.
\end{notation}
\begin{theorem}\label{slices_MSL}
    Let $F$ be a field and $\Lambda=\Z[\nicefrac{1}{e}]$, where $e$ is the exponential characteristic of $F$.
    \begin{enumerate}
    \item As $\EM\Lambda$-modules, the slices of $\MSL_\Lambda$ are given as follows: 
    \[\slice_q(\MSL_\Lambda)\cong \Sigma^{q+(q)}\EM\Lambda\otimes\cyc_{2q}(\W)\oplus \bigoplus_{p=1}^q \Sigma^{q-p+(q)}\EM\Lambda\otimes\hmlg_{2q-2p}(\W).\]
    \item Let $N_q:=\bigoplus_{p=2}^{q} \Sigma^{q-p+(q)}\EM\Lambda\otimes\hmlg_{2q-2p}(\W)$ and $\varphi:= \iota_{2q-2}^*(\delta_{2q})\colon\W_{2q}\to \cyc_{2q-2}(\W)$. 
    The morphism $\MSL_\Lambda\to \MWL_\Lambda$ induces the map 
    \[\Sigma^{q+(q)}\EM\Lambda\otimes\cyc_{2q}(\W)\oplus\Sigma^{q-1+(q)}\EM\Lambda\otimes\hmlg_{2q-2}(\W)\oplus N_q\xrightarrow{\Sigma^{q+(q)}(\EM\Lambda\otimes\mathrm{inc}\ \partial\ 0)} \Sigma^{q+(q)}\EM\Lambda\otimes\W_{2q} \] 
    on the $q$-th slice, where $\mathrm{inc}$ and $\partial$ are as in Lemma \ref{lemma_EM_map}.
    \end{enumerate}
\end{theorem}
\begin{proof}
    The proof uses induction. 
    The cofiber sequence \ref{MWL_cofib}(1) induces on slices a cofiber sequence 
    \begin{equation}\label{eq:proof-slices-msl}
    \slice_q(\MSL_\Lambda)\to \slice_q(\MWL_\Lambda)\xrightarrow{\slice_q(d)}\slice_q(\Sigma^{1+(1)}\MSL_\Lambda)\cong\Sigma^{1+(1)}\slice_{q-1}(\MSL_\Lambda)
    \end{equation}
    of $\EM\Lambda$-modules. 
    The statements for $q=0$ follow from the effectivity of $\MSL$ and $\slice_0(\MWL_\Lambda)\cong\EM\Lambda$; see Theorem \ref{slices_MWL}, where $\cyc_0(\W)=\W_0\cong \Z$. 
    Assuming that both statements are proven for $q-1$, the cofiber sequence~\ref{eq:proof-slices-msl} reduces the task to identifying the fiber of $\slice_q(d)$. 
    Theorem \ref{slices_MWL} and the induction hypothesis identify $\slice_q(d)$ with 
    \[(\phi\ \psi_1\ \dots\  \psi_{q-1})^{\mathrm{tr}}\colon\Sigma^{q+(q)}\EM\Lambda\otimes\W_{2q}\to \Sigma^{q+(q)}\EM\Lambda\otimes\cyc_{2q-2}(\W)\oplus \bigoplus_{p=1}^{q-1} \Sigma^{q-p+(q)}\EM\Lambda\otimes\hmlg_{2q-2-2p}(\W)\] 
    for some morphisms $\phi$ and $\psi_1,\dots,\psi_{q-1}$. 
    Since after $q+(q)$-desuspension $\psi_p$ belongs to the group \[[\EM\Lambda\otimes\W_{2q},\Sigma^{-p}\EM\Lambda\otimes\hmlg_{2q-2p-2}(\W)]_{\EM\Lambda\text{-}\mathbf{mod}}\] 
    which is trivial for $p>0$ by Lemma \ref{lemma_hom_(s,0)}, all $\psi_p$ are zero.
    Hence, the $q$-th slice of $\MSL_\Lambda$ is given by 
    \[\slice_q(\MSL_\Lambda)\cong\mathrm{fib}(\phi)\oplus N_q,\] 
    and it remains to identify $\phi$. 
    For this purpose consider the following diagram
    \[\begin{tikzcd}
	{\Sigma^{q+(q)}\EM\Lambda\otimes\W_{2q}} & {\Sigma^{q+(q)}\EM\Lambda\otimes\cyc_{2q-2}(\W)} \\
    & {\Sigma^{q+(q)}\EM\Lambda\otimes\W_{2q-2},}
	\arrow["\phi",from=1-1, to=1-2]
	\arrow[from=1-1, to=2-2]
    \arrow["\Sigma^{q+(q)}\EM\Lambda\otimes\mathrm{inc}",from=1-2, to=2-2]
    \end{tikzcd}\] 
    where the diagonal morphism is given by $\slice_q(\delta)$ (up to identifications from Theorem \ref{slices_MWL}). 
    This diagram commutes in $\EM\Lambda\text{-}\mathbf{mod}$ by the definition of $\phi$ and the induction hypothesis for the second assertion. 
    Combining this with the injectivity of $\mathrm{inc}\colon\cyc_{2q-2}(\W)\hookrightarrow\W_{2q-2}$, the map $\phi$ is thus given by $\Sigma^{q+(q)}\EM\Lambda\otimes\iota_{2q-2}^*(\delta_{2q})$. 
    The induction steps for both statements follow from Lemma~\ref{lemma_EM_map}.
\end{proof}
\begin{remark}\label{remark_ANSS}
    The groups that appear in the slices of $\MSL_\Lambda$ are precisely the groups constituting the $E_2$-page of the Adams--Novikov spectral sequence for $\mathbf{MSU}$ by \cite[Proposition 5.3]{CLP}: 
    \[ \slice_q(\MSL_\Lambda)\cong \bigoplus_{p=0}^q \Sigma^{q-p+(q)}\EM\Lambda\otimes\Ext^{p,2q}_{\mathbf{MU^*(MU)}}(\mathbf{MU^*(MSU),MU^*}). \]
\end{remark}
\begin{remark}\label{remark:first_slices}
Comparing the previous theorem with Theorem \ref{thm_conner_floyd}, the first few slices $\slice_q(\MSL_\Lambda)$ can be given in an explicit form.
The result for $q\leq 6$ looks as follows:
\allowdisplaybreaks
\begin{align*} 
        \slice_0(\MSL_\Lambda) & \cong \EM\Lambda\ \text{(cf. \cite[Example 16.35]{BH-Norms}, \cite[Corollary 4.9]{NanMSL}),} \\
        \slice_1(\MSL_\Lambda) & \cong \Sigma^{(1)}\EM\Lambda/2\ \text{(cf. \cite[Corollary 4.10]{NanMSL}),} \\
        \slice_2(\MSL_\Lambda) & \cong \Sigma^{(2)}\EM\Lambda/2 \oplus \Sigma^{2+(2)}\EM\Lambda, \\
        \slice_3(\MSL_\Lambda) & \cong \Sigma^{(3)}\EM\Lambda/2 \oplus  \Sigma^{2+(3)}\EM\Lambda/2 \oplus \Sigma^{3+(3)}\EM\Lambda, \\
        \slice_4(\MSL_\Lambda) & \cong \Sigma^{(4)}\EM\Lambda/2 \oplus \Sigma^{2+(4)}\EM\Lambda/2 \oplus \Sigma^{4+(4)}\EM\Lambda^2, \\
        \slice_5(\MSL_\Lambda) & \cong \Sigma^{(5)}\EM\Lambda/2 \oplus \Sigma^{2+(5)}\EM\Lambda/2 \oplus \Sigma^{4+(5)}\EM\Lambda/2 \oplus \Sigma^{5+(5)}\EM\Lambda^2, \\
        \slice_6(\MSL_\Lambda) & \cong \Sigma^{(6)}\EM\Lambda/2 \oplus \Sigma^{2+(6)}\EM\Lambda/2 \oplus \Sigma^{4+(6)}\EM\Lambda/2 \oplus \Sigma^{6+(6)}\EM\Lambda^4.
        %\slice_7(\MSL_\Lambda) & \cong \Sigma^{(7)}\EM\Lambda/2 \oplus \Sigma^{2+(7)}\EM\Lambda/2 \oplus \Sigma^{4+(7)}\EM\Lambda/2 \oplus \Sigma^{6+(7)}\EM\Lambda/2 \oplus \Sigma^{7+(7)}\EM\Lambda^4, \\
        %\slice_8(\MSL_\Lambda) & \cong \Sigma^{(8)}\EM\Lambda/2 \oplus \Sigma^{2+(8)}\EM\Lambda/2 \oplus \Sigma^{4+(8)}\EM\Lambda/2 \oplus \Sigma^{6+(8)}\EM\Lambda/2 \oplus \Sigma^{8+(8)}\EM\Lambda^7, \\
        %\slice_9(\MSL_\Lambda) & \cong \Sigma^{(9)}\EM\Lambda/2 \oplus \Sigma^{2+(9)}\EM\Lambda/2 \oplus \Sigma^{4+(9)}\EM\Lambda/2 \oplus \Sigma^{6+(9)}\EM\Lambda/2 \oplus \Sigma^{8+(9)}\EM(\Lambda/2)^2 \oplus \Sigma^{9+(9)}\EM\Lambda^8.
\end{align*}
Since $\Lambda[\nicefrac{1}{2}]\otimes
\hmlg_{*}(\W)=0$, inverting $2$ produces a simpler picture, namely the following equivalences: 
\[\slice_q(\MSL_{\Lambda[\nicefrac{1}{2}]})\cong\Sigma^{q+(q)}\EM\Lambda[\nicefrac{1}{2}]\otimes\cyc_{2q}(\W)\cong\Sigma^{q+(q)}\EM\Lambda[\nicefrac{1}{2}]^{p(q)-p(q-1)}\]
This simple form of the slices away from $2$ corresponds to the classical fact that the Adams--Novikov spectral sequence for $\mathbf{MSU}$ collapses at the $E_2$-page away from $2$. 
\end{remark}
The second assertion of Theorem~\ref{slices_MSL} allows us to deduce the action of the first motivic Hopf element $\eta\in\pi_{0+(1)}(\sph)$ on the slices of $\MSL_{\Lambda}$ as follows.
\begin{lemma}\label{lemma_eta_action}
    For every non-negative integer $q$, the morphism \[\Sigma^{(1)}\slice_{q-1}(\MSL_\Lambda)\cong\slice_q(\Sigma^{(1)}\MSL_\Lambda)\xrightarrow{\slice_q(\eta)}\slice_q(\MSL_\Lambda)\] 
    \begin{enumerate}
        \item restricts to the morphism \[\Sigma^{q-1+(q)}\EM\Lambda\otimes\cyc_{2q-2}(\W)\xrightarrow{\Sigma^{q-1+(q)}\EM\Lambda\otimes\mathrm{pr}_{2q-2}} \Sigma^{q-1+(q)}\EM\Lambda\otimes\hmlg_{2q-2}(\W),\] where the epimorphism $\mathrm{pr}_{2q-2}\colon\cyc_{2q-2}(\W)\twoheadrightarrow\hmlg_{2q-2}(\W)$ is the quotient map;
        \item restricts to the identity map \[  \Sigma^{q-2+(q)}\EM\Lambda\otimes\hmlg_{2q-4}(\W)\oplus\Sigma^{(1)}N_{q-1}\to N_q;\]
        \item restricts to the zero morphism on other direct summands.
    \end{enumerate}
\end{lemma}
\begin{proof}
    Statements (1) and (2) follow directly from part (2) of Theorem \ref{slices_MSL}. 
    Lemma \ref{lemma_hom_(s,0)} implies the last statement.
\end{proof}
This information suffices to determine the slices of the $\eta$-periodic special linear algebraic cobordism spectrum $\MSL_\Lambda[\eta^{-1}]$, defined as the (homotopy) colimit \[\MSL_\Lambda\xrightarrow{\Sigma^{(-1)}\eta}\Sigma^{(-1)}\MSL_\Lambda\xrightarrow{\Sigma^{(-2)}\eta}\Sigma^{(-2)}\MSL_\Lambda\xrightarrow{\Sigma^{(-3)}\eta}\cdots.\]
\begin{theorem}\label{slice_MSL_eta_inv}
    Let $F$ be a field and $\Lambda=\Z[\nicefrac{1}{e}]$, where $e$ is the exponential characteristic of $F$.
    As $\EM\Lambda$-modules, the slices of $\MSL_\Lambda[\eta^{-1}]$ are given as follows: \[\slice_q(\MSL_\Lambda[\eta^{-1}])\cong \bigoplus_{p=0}^{\infty} \Sigma^{p+(q)}\EM\Lambda\otimes\hmlg_{2p}(\W)\]
\end{theorem}
\begin{proof}
    By the definition of $\MSL_\Lambda[\eta^{-1}]$, the multiplication by the $q$-th power of the first motivic Hopf element induces an equivalence on $\MSL_\Lambda[\eta^{-1}]$. 
    Taking the induced morphism on the $q$-th slices induces the equivalence \[\Sigma^{(q)}\slice_0(\MSL_\Lambda[\eta^{-1}])\cong\slice_{q}(\Sigma^{(q)}\MSL_\Lambda[\eta^{-1}])\xrightarrow[\slice_q(\eta^{q})]{\simeq}\slice_{q}(\MSL_\Lambda[\eta^{-1}]).\] 
    Therefore, it is enough to prove the claim for $q=0$. 
    Since slices commute with (homotopy) colimits by \cite[Corollary 4.5]{SpiMGL1}, then
    \begin{gather*} \slice_0(\MSL_\Lambda[\eta^{-1}])=\slice_0(\colim_n \Sigma^{(-n)}\MSL_\Lambda)\cong \colim_n \Sigma^{(-n)} \slice_n(\MSL_\Lambda)\cong \\ \cong\colim_{n} \Bigl(\Sigma^{n}\EM\Lambda\otimes \cyc_{2n}(\W)\oplus \bigoplus_{p=1}^{n} \Sigma^{n-p}\EM\Lambda\otimes\hmlg_{2n-2p}(\W)\Bigr),
    \end{gather*}
    and it remains to identify morphisms in this direct system. 
    This has been done in Lemma~\ref{lemma_eta_action} (up to appropriate $\Gm$-suspensions). 
    It follows that the desired colimit decomposes into the infinite direct sum of the colimits of the following form: \[\Sigma^{s}\EM\Lambda\otimes\cyc_{2s}(\W)\xrightarrow{\Sigma^{s}\EM\Lambda\otimes\mathrm{pr}_{2s}}\Sigma^{s}\EM\Lambda\otimes\hmlg_{2s}(\W)\xrightarrow{\id}\Sigma^{s}\EM\Lambda\otimes\hmlg_{2s}(\W)\xrightarrow{\id}\cdots\] 
    Each of them obviously gives an additional summand $\Sigma^{s}\EM\Lambda\otimes\hmlg_{2s}(\W)$. 
    This concludes the proof.
\end{proof}
\begin{remark}
    In fact, all the above results can be easily generalized to an arbitrary compatible pair $(S,\Lambda)$ using cellularity of the cobordism spectra $\MWL$ and $\MSL$ (see \cite[Corollary 2.17 and Remark 2.18]{RSO19}).
\end{remark}
\subsection{Comparison with very effective hermitian $K$-theory}\label{comp_slices_msl_kq}
Let $\kq$ denote the very effective cover of the hermitian $K$-theory spectrum $\mathbf{KQ}$. 
By construction, there is a canonical morphism $\kq\to\mathbf{KQ}$, even over fields of characteristic $2$, or arbitrary Dedekind domains. 
The computation of the slices of $\kq$  obtained in \cite{very_eff_KQ} (see also \cite{KRO_hermitian_k_theory_char2} for the case of arbitrary characteristic) will be restated in convenient terms, in order to compare them easily with the slices of $\MSL_\Lambda$. 
Consider the Wood cofiber sequence for $\kq$ (see \cite[Proposition 2.11]{very_eff_KQ} and \cite[Theorem 4.3]{KRO_hermitian_k_theory_char2} in characteristic $2$): 
\begin{equation}   \Sigma^{(1)}\kq\xrightarrow{\eta}\kq\xrightarrow{\mathrm{forg}}\kgl \xrightarrow{\mathrm{hyp}}\Sigma^{1+(1)}\kq
\end{equation}
It yields the morphism $\Sigma^{1+(1)}\mathrm{forg}\circ\mathrm{hyp}\colon\kgl\to\Sigma^{1+(1)}\kgl$, which is an analogue of $\delta\colon \MWL\to \Sigma^{1+(1)}\MWL$ in this case. 
The square of this morphism is zero in the sense that \[\Sigma^{1+(1)}(\Sigma^{1+(1)}\mathrm{forg}\circ\mathrm{hyp})\circ (\Sigma^{1+(1)}\mathrm{forg}\circ\mathrm{hyp})=0.\] 
Hence it induces a chain complex of abelian groups 
\[ \cdots\to\pi_{p+1+(q+1)}(\kgl)\to\pi_{p+(q)}(\kgl)\to\pi_{p-1+(q-1)}(\kgl)\to\cdots.\]
Restricting attention to the geometric diagonal $\pi_{\star+(\star)}(\kgl)\cong\Z[\beta]$ defines the following chain complex 
\[\cdots\to\Z\{\beta^{s+1}\}\to\Z\{\beta^s\}\to\Z\{\beta^{s-1}\}\to\cdots\] 
which is indexed by powers of the Bott periodicity class $\beta$.
It does not depend on the chosen base field. 
In fact, for any field it identifies with the analogous chain complex in topology obtained from the connective complex $K$-theory spectrum $\mathbf{ku}$.
\begin{notation}
    Let $\cyc_{2s}(\Z[\beta])$ (resp. $B_{2s}(\Z[\beta])$, resp. $\hmlg_{2s}(\Z[\beta])$) denote the group of cycles (resp. boundaries, resp. homology) of the chain complex $(\Z[\beta],(\Sigma^{1+(1)}\mathrm{forg}\circ\mathrm{hyp})_*)$ in the spot corresponding to $\Z\{\beta^s\}$. 
    By convention, these groups are trivial for odd indices: $\cyc_{2s+1}(\Z[\beta])=0$, $B_{2s+1}(\Z[\beta])=0$, $\hmlg_{2s+1}(\Z[\beta])=0$.
\end{notation}
\begin{theorem}
    The slices of the very effective hermitian $K$-theory spectrum $\kq$ are given by \[ \slice_q(\kq)\cong\Sigma^{q+(q)}\EM\Z\otimes\cyc_{2q}(\Z[\beta])\oplus \bigoplus_{p=1}^q \Sigma^{q-p+(q)}\EM\Z\otimes\hmlg_{2q-2p}(\Z[\beta]). \]
\end{theorem}
\begin{proof}
    The same arguments as in the proof of Theorem \ref{slices_MSL}, using the Wood cofiber sequence instead of the cofiber sequence (1) of Theorem \ref{MWL_cofib}, show the claim.
\end{proof}
\begin{remark}
    It is easy to compute the groups $\cyc_{2s}(\Z[\beta])$ and $\hmlg_{2s}(\Z[\beta])$ explicitly. 
    Indeed, the group of cycles is trivial if $s$ is odd and isomorphic to $\Z$ if $s$ is even. 
    In turn, the respective homology group is zero if $s$ is odd and isomorphic to $\Z/2$ if $s$ is even, since the dimension of the hyperbolic plane is $2$. 
    Combining this computation with the previous theorem, one obtains \cite[Theorem 3.2]{very_eff_KQ} and \cite[Corollary 4.4]{KRO_hermitian_k_theory_char2}.
\end{remark}
An explicit verification shows that the map $f\colon\W_{2*}\to\Z[\beta]$ defined in Lemma \ref{lemma_map_s_mwl_to_kgl} induces a homomorphism of chain complexes as follows:
\begin{center}
\begin{tikzcd}
\cdots \ar[r] & \W_{2q+2} \ar[r] \ar[d,"f_{2q+2}"] & \W_{2q} \ar[r] \ar[d,"f_{2q}"] & \W_{2q-2} \ar[r] \ar[d,"f_{2q-2}"] & \cdots \\ \cdots \ar[r] & \Z\{\beta^{q+1}\} \ar[r] & \Z\{\beta^q\} \ar[r] & \Z\{\beta^{q-1}\} \ar[r] & \cdots.
\end{tikzcd}
\end{center}

\begin{notation}\label{not.cyc-hom-f}
    The associated homomorphisms between cycle and homology groups are denoted \[\cyc_{2q}(f)\colon\cyc_{2q}(\W)\to\cyc_{2q}(\Z[\beta]) \quad \mathrm{and} \quad \hmlg_{2q}(f)\colon\hmlg_{2q}(\W)\to\hmlg_{2q}(\Z[\beta]).\]
\end{notation}
\begin{proposition}\label{prop:slices-msl-kq}
    The map on the $q$-th slices $\slice_q(\MSL_\Lambda)\to \slice_q(\kq_\Lambda)$ 
    \begin{enumerate}
        \item restricts to the map
        \[\Sigma^{q+(q)}\EM\Lambda\otimes\cyc_{2q}(f)\colon\Sigma^{q+(q)}\EM\Lambda\otimes\cyc_{2q}(\W)\to\Sigma^{q+(q)}\EM\Lambda\otimes\cyc_{2q}(\Z[\beta]),\] 
        where the homomorphism $\cyc_{2q}(f)\colon\cyc_{2q}(\W)\to\cyc_{2q}(\Z[\beta])$ is discussed in Notation~\ref{not.cyc-hom-f};
        \item restricts to the map 
        \[\Sigma^{q-p+(q)}\EM\Lambda\otimes\hmlg_{2q-2p}(f)\colon\Sigma^{q-p+(q)}\EM\Lambda\otimes\hmlg_{2q-2p}(\W)\to\Sigma^{q-p+(q)}\EM\Lambda\otimes\hmlg_{2q-2p}(\Z[\beta]),\] 
        where the homomorphism $\hmlg_{2q-2p}(f)\colon\hmlg_{2q-2p}(\W)\to\hmlg_{2q-2p}(\Z[\beta])$ is discussed in Notation~\ref{not.cyc-hom-f};
        \item restricts to the zero map on other direct summands.
    \end{enumerate}
\end{proposition}
\begin{proof}
    This follows immediately from the definitions, Lemma \ref{lemma_map_s_mwl_to_kgl} and the chosen approach to compute the slices.
\end{proof}
For the sake of completeness, we formulate the computation of the slices of connective Witt theory $\mathbf{kw}$ in these terms and its comparison with the slices of $\MSL_\Lambda[\eta^{-1}]$. 
\begin{theorem}
    The slices of the connective Witt theory spectrum $\mathbf{kw}$ are given by (cf. \cite[Theorem 3.4]{very_eff_KQ}) 
    \[\slice_q(\mathbf{kw})\cong \bigoplus_{p=0}^\infty \Sigma^{p+(q)}\EM\Z\otimes\hmlg_{2p}(\Z[\beta]).\] 
    Moreover, for every nonnegative integer $q$, the canonical morphism $\slice_q(\MSL_\Lambda[\eta^{-1}])\to\slice_q(\mathbf{kw}_\Lambda)$ is given by the direct sum of the following maps: 
    \[\Sigma^{p+(q)}\EM\Lambda\otimes \hmlg_{2p}(f)\colon\Sigma^{p+(q)}\EM\Lambda\otimes\hmlg_{2p}(\W)\to\Sigma^{p+(q)}\EM\Lambda\otimes\hmlg_{2p}(\Z[\beta]).\]
\end{theorem}
\begin{proof}
    The same arguments as in Lemma \ref{lemma_eta_action} and Theorem \ref{slice_MSL_eta_inv} compute the slices of $\mathbf{kw}=\kq[\eta^{-1}]$. 
    The last claim follows from Proposition \ref{prop:slices-msl-kq}.
\end{proof}
\section{First slice differentials for some cobordism spectra}\label{section:4}
The purpose of this section is to determine a few of the first slice differentials for $\MSL$. 
Starting from this section, we assume that the characteristic of the chosen base field $F$ is different from $2$.
The case of fields $F$ with $\chark(F)=2$ is considered in Appendix \ref{appendix_char2}.
We proceed backwards along the maps
$\MSL\to\MSL\to \MGL$, beginning with the identification of all first slice differentials for $\MGL_\Lambda$, which is based on the following lemma.
\begin{lemma}\label{lemma:d1-integralsummands}
The group of maps of the form $\EM\Lambda\to\Sigma^{2+(1)}\EM\Lambda$ is isomorphic to $\mathbb{F}_2\{\partial^2_\infty\mathsf{Sq}^2\mathrm{pr}^\infty_2\}$.
\end{lemma}
\begin{proof}
    The cofiber sequences 
    \[ \effcover_{q+1} \sph \to \effcover_q\sph \to \slice_q \sph \to \Sigma \effcover_{g+1} \sph\]
    for $q\in \{0,1\}$, the vanishing $[\Sigma^q\sph,\Sigma^{2+(1)}\EM\Lambda]=0$ and the vanishing $[\effcover_2\sph,\Sigma^{2+(1)}\EM\Lambda]=0$ imply that the first slice differential $\slice_0\sph \to \Sigma\slice_1\sph$ induces an isomorphism
    \[ [\slice_0\sph,\Sigma^{2+(1)}\EM\Lambda] \cong [\Sigma\slice_1\sph,\Sigma^{2+(1)}\EM\Lambda].\]
    Since $\slice_0\sph\cong \EM\Z$ by \cite[Theorem B.4]{BH-Norms}, $\slice_1\sph_\Lambda\cong \Sigma^{(1)}\EM\Lambda/2$ and the first slice differential $\slice_0(\sph)\to \Sigma\slice_1(\sph)$ equals $\mathsf{Sq}^2\mathrm{pr}^\infty_2$ by \cite[Lemma 4.1]{RSO19}, the result follows with Lemma~\ref{lemma_hom_(s,0)}.
\end{proof}

We use the usual presentation of the Lazard ring
\[ \MU_*\cong\Z[a_1,a_2,\dots] \]
where the first polynomial generators $a_i$ are given by the following linear combinations of the coefficients $a_{ij}$ of the universal formal group law: $a_1=a_{11}$, $a_2=a_{12}$, $a_3=a_{22}-a_{13}$, $a_4=a_{14}$.
\begin{theorem}\label{thm.firstslicediffMGL}
    The restriction of the slice differential $\dd_1^{\MGL_\Lambda}(q)\colon\slice_q(\MGL_\Lambda)\to \Sigma\slice_{q+1}(\MGL_\Lambda)$ to 
    \[ \Sigma^{q+(q)}\EM\Lambda\{a_I\}\to\Sigma^{q+2+(q+1)}\EM\Lambda\{a_J\}\]
    is given by $\partial^2_\infty\mathsf{Sq}^2\mathrm{pr}^\infty_2$ if $a_J=a_1a_I$ and is trivial otherwise.
\end{theorem}
\begin{proof}
    The desired component of the first differential is either trivial or given by $\partial^2_\infty\mathsf{Sq}^2\mathrm{pr}^\infty_2$ according to Lemma~\ref{lemma:d1-integralsummands}.
    Over $\RR$ this component is non-trivial if and only if it maps $\partial^{2}_\infty(\tau^3)\in H^{1,3}(\RR)$ to a non-trivial $2$-torsion element in $H^{4,4}(\RR)\cong\KMil_4(\RR)$, as the following computation shows:
    \[ (\partial^2_\infty\mathsf{Sq}^2\mathrm{pr}^\infty_2)\circ\partial^2_\infty(\tau^3)=\partial^2_\infty\mathsf{Sq}^2\mathsf{Sq}^1(\tau^3)=\partial^2_\infty\mathsf{Sq}^2(\tau^2\rho)=\partial^2_\infty(\tau\rho^3) = \{-1,-1,-1,-1\}\neq 0. \]
    To compute the first slice differential on the elements $\partial^2_\infty(\tau^3)a_I$, apply the Leibniz rule given in \cite[Proposition 2.24]{RSO19} to the multiplicative slice spectral sequence for $\MGL_\Lambda$:
    \[ \dd_1^{\MGL_\Lambda}(\partial^2_\infty(\tau^3) a_I)=\dd_1^{\MGL_\Lambda}(\partial^2_\infty(\tau^3)) a_I+\partial^2_\infty(\tau^3)\dd_1^{\MGL_\Lambda}(a_I)=\dd_1^{\MGL_\Lambda}(\partial^2_\infty(\tau^3)) a_I.\]
    Here the vanishing $\dd_1^{\MGL_\Lambda}(a_I)=0$, which holds for degree reasons, is used. 
    It remains to note that the image of $\partial^2_\infty(\tau^3)$ under the first slice differential is given by $\partial^2_\infty(\tau\rho^3)a_1$. 
    This follows from the computation of $\dd_1^\mathbf{KGL}$ \cite[Lemma 5.1]{RO16} and $\slice_i(\MGL_\Lambda)\cong\slice_i(\mathbf{KGL}_\Lambda)$ for $i=0,1$.

    The motivic Steenrod algebra is suitably invariant under base change, as \cite[Remark 10.27, Theorem 10.26]{Spi_HZ} shows.
    Also the slices of $\MGL_\Lambda$ pull back.
    Hence the identification of the first slice differential extends from $\RR$ to any subring in which $2$ is invertible, such as $\QQ$ and $\Z_{(p)}$ for $p$ an odd prime.
    Further base change to the residue field $\mathbb{F}_p$ of $\Z_{(p)}$ provides the identification for all prime fields of characteristic not $2$. 
    The conclusion over an arbitrary field $F$ of characteristic different from $2$ follows by base change from the prime subfield $F_0\subset F$.
\end{proof}
\begin{notation}
    The resulting map is denoted as follows:
    \[\dd_1^{\MGL_\Lambda}(q)=\partial^2_\infty\mathsf{Sq}^2\mathrm{pr}^\infty_2\otimes a_1\colon\Sigma^{q+(q)}\EM\Lambda\otimes\MU_{2q}\to\Sigma^{q+2+(q+1)}\EM\Lambda\otimes\MU_{2q+2}\]
\end{notation}
Our next goal is to compute the first few $\dd_1$-differentials for $\MWL_\Lambda$. 
The graded group $\W_*$ admits a non-canonical ring structure as a graded polynomial ring 
\[ \W_*\cong\Z[x_1,x_k\,|\,k\geq 3],\ \text{where}\ \mathrm{deg}(x_k)=2 \]
as explained in \cite[Theorem 6.10]{CLP}.
Moreover, the image of $x_1$ in $\MU_2$ is equal to $a_1$. 
However, the inclusion $\W_*\to\MU_*$ does not preserve the ring structure. 
To avoid confusion, we denote the multiplication in $\W_*$ by $*$. 
Since $\W_*\to\MU_*$ is injective, we identify elements of $\W_*$ with their images. 
The following tedious calculations will be useful.
\begin{lemma}\label{lemma:xi_decomposition_ai}
    The generator $x_3\in\W_6$ identifies with $-3a_3+2a_1a_2\in \MU_6$, and the generator $x_4\in\W_8$ identifies with $-a_2^2-2a_4-9a_1a_3+7a_1^2a_2\in \MU_8$.
\end{lemma}
\begin{proof}
    The element $x_3$ is equal to $-a_{22}\in\MU_6$ by \cite[Proposition 7.1]{Bak_HHA}.
    The formula for $x_3$ then follows from the relation $2a_{11}a_{12}+3a_{13}-2a_{22}=0$ in the Lazard ring. 
    To prove the second formula, first note that the generator $y_4$ of $\MSU_8$ is given by $2x_4-x_1*x_3$ \cite[Theorem 7.1 and below]{CLP}.
    Away from $2$, $y_4\in\MSU_8\textstyle[\frac{1}{2}]$ decomposes as $-a_{23}+\textstyle \frac{3}{2}a_1x_3$ in $\MU_8\textstyle[\frac{1}{2}]$ by \cite[Proposition 7.1]{Bak_HHA}. 
    From the relation $2a_{12}^2+3a_{11}a_{13}+4a_{14}-2a_{23}=0$ in $\MU_*$, together with the values of the Milnor numbers $s_4(a_4)=\pm m_4=\pm 5$ and $s_4(y_4)=\pm 2m_4m_3=\pm 5\cdot 4$, one deduces that $y_4=-2a_{23}+3a_1x_3\in\MU_8$. 
    As a consequence,
    \[ x_4=\textstyle\frac{1}{2}(y_4+x_1x_3)=\frac{1}{2}(-2a_{23}+4a_1x_3)=-a_{23}+2a_1x_3\in\MU_8,\]
    and the claim follows from the formula for $x_3$ already proven and the relations on the coefficients of the universal formal group law mentioned above. 
\end{proof}
\begin{lemma}\label{lemma:xi_decomposition_ai_2}
    The decomposable generators of $\W_{\leq 8}$ are given by the following linear combinations of the $a_I$:
    \begin{align*}
        & x_1^{*2}=a_1^2+8a_2; \\
        & x_1^{*3}=a_1^3+8a_1a_2; \\ 
        & x_1^{*4}=a_1^4+16a_1^2a_2+64a_2^2; \\
        & x_1*x_3=-3a_1a_3+2a_1^2a_2.
    \end{align*}
\end{lemma}
\begin{proof}
    It is a rather straightforward calculation that uses Lemma \ref{lemma:xi_decomposition_ai} and the formula \[a*b=ab+2a_2\delta(a)\delta(b)\] 
    stated in \cite[Proposition 6.8]{CLP} (note that $\delta=-\partial$). 
    The action of the operation $\delta$ on $x_1$ and $x_3$ is given by $\delta(x_1)=-2$ and $\delta(x_3)=0$.
\end{proof}
\begin{lemma}\label{lem:d1MWL}
    In slice degrees $0\leq q\leq 3$ the differential $\dd_1^{\MWL_\Lambda}(q)$ has the following form:
    \[ \dd_1^{\MWL_\Lambda}(q)=\partial^2_\infty\mathsf{Sq}^2\mathrm{pr}^\infty_2\otimes x_1\colon\Sigma^{q+(q)}\EM\Lambda\otimes\W_{2q}\to\Sigma^{q+2+(q+1)}\EM\Lambda\otimes\W_{2q+2}\]
\end{lemma}
\begin{proof}
    For $q=0$ the claim follows immediately from the case of $\MGL_\Lambda$, since $\slice_i(\MWL_\Lambda)\cong \slice_i(\MGL_\Lambda)$ for $i\in\{0,1\}$. 
    Consider the commutative diagram
    \begin{center}
    \begin{tikzcd}[ampersand replacement=\&]
    \EM\Lambda\{x_1\} \arrow[rr, "\simeq"] \arrow[d, "\dd_1^{\MWL_\Lambda}(1)"'] \& \& \EM\Lambda\{a_1\} \arrow[d, "\dd_1^{\MGL_\Lambda}(1)"]               \\
    {\Sigma^{2+(1)}\EM\Lambda\{x_1^{*2}\}} \arrow[rr, "{\renewcommand{\arraystretch}{0.9}
    \begin{pmatrix} 1 \\ 8 \end{pmatrix}}"] \arrow[d, "\dd_1^{\MWL_\Lambda}(2)"'] \& \&  {\Sigma^{2+(1)}\EM\Lambda\{a_1^2\}\oplus\Sigma^{2+(1)}\EM\Lambda\{a_2\}} \arrow[d, "\dd_1^{\MGL_\Lambda}(2)"] \\
    {\Sigma^{4+(2)}\EM\Lambda\{x_1^{*3}\}\oplus\Sigma^{4+(2)}\EM\Lambda\{x_3\}} \arrow[rr, "{\renewcommand{\arraystretch}{0.9}
    \setlength{\arraycolsep}{3pt}\begin{pmatrix} 1 & 0 \\ 8 & 2 \\ 0 & -3 \end{pmatrix}}"] \& \& {\Sigma^{4+(2)}\EM\Lambda\{a_1^{3}\}\oplus\Sigma^{4+(2)}\EM\Lambda\{a_1a_2\}\oplus\Sigma^{4+(2)}\EM\Lambda\{a_3\}},
    \end{tikzcd}
    \end{center}
    which is induced by the morphism $\mathrm{forg}^\prime\colon \MWL_\Lambda\to\MGL_\Lambda$. 
    The horizontal maps have this form by Lemma \ref{lemma:xi_decomposition_ai} and Lemma \ref{lemma:xi_decomposition_ai_2}. 
    To see that $\dd_1^{\MWL_\Lambda}(1)$ is given by the unique non-trivial arrow, compose the top square with the projection onto $\Sigma^{2+(1)}\EM\Lambda\{a_1^2\}$ and use $\dd_1^{\MGL_\Lambda}(1)=(\partial^2_\infty\mathsf{Sq}^2\mathrm{pr}^\infty_2\ 0)^{\mathrm{tr}}$. 
    The formula for the differential $\dd_1^{\MWL_\Lambda}(2)$ follows in the same manner, composing the bottom square with the projections onto $\Sigma^{4+(2)}\EM\Lambda\{a_1^3\}$ and $\Sigma^{4+(2)}\EM\Lambda\{a_3\}$. 
    
    In slice degree $3$, the desired differential $\dd_1^{\MWL_\Lambda}(3)$ is the $3+(3)$-suspension of
    \[\begin{pmatrix}
        \alpha_1 & \beta_1 \\ \alpha_2 & \beta_2 \\ \alpha_3 & \beta_3
    \end{pmatrix}\colon\EM\Lambda\{x_1^{*3}\}\oplus\EM\Lambda\{x_3\}\to \Sigma^{2+(1)}\EM\Lambda\{x_1^{*4}\}\oplus\Sigma^{2+(1)}\EM\Lambda\{x_1*x_3\}\oplus\Sigma^{2+(1)}\EM\Lambda\{x_4\} \]
    for $\alpha_i,\,\beta_i\in\mathbb{F}_2\{\partial^2_\infty\mathsf{Sq}^2\mathrm{pr}^\infty_2\}$.
    In turn, Lemma \ref{lemma:xi_decomposition_ai} and Lemma \ref{lemma:xi_decomposition_ai_2} again provide a complete description of the morphism $\slice_4(\MWL_\Lambda)\to\slice_4(\MGL_\Lambda)$. 
    It is given by applying $\Sigma^{4+(4)}$ to the following matrix:
    \begin{align*} & \begin{pmatrix}
        1 & 0 & 0 \\ 16 & 2 & 7 \\ 64 & 0 & -1 \\ 0 & -3 & -9 \\ 0 & 0 & -2
    \end{pmatrix}\colon\EM\Lambda\{x_1^{*4}\}\oplus\EM\Lambda\{x_1*x_3\}\oplus\EM\Lambda\{x_4\}\to \\
    & \hspace{5.5cm} \to\EM\Lambda\{a_1^4\}\oplus\EM\Lambda\{a_1^2a_2\}\oplus\EM\Lambda\{a_2^2\}\oplus\EM\Lambda\{a_1a_3\}\oplus\EM\Lambda\{a_4\}.\end{align*} 
    A commutative diagram similar to the one above and the computation from Theorem~\ref{thm.firstslicediffMGL} of the first slice differential for $\MGL_\Lambda$ lead to the following equality:
    \[\begin{pmatrix}
        \alpha_1 & \beta_1 \\  7\alpha_3 & 7\beta_3 \\ -\alpha_3 & -\beta_3 \\ -3\alpha_2-9\alpha_3 & -3\beta_2-9\beta_3 \\0 & 0
    \end{pmatrix} = 
    \begin{pmatrix}
        \partial^2_\infty\mathsf{Sq}^2\mathrm{pr}^\infty_2 & 0 \\ 0 & 0 \\ 0 & 0 \\ 0 & -3\partial^2_\infty\mathsf{Sq}^2\mathrm{pr}^\infty_2 \\ 0 & 0
    \end{pmatrix}\]
    Here we used several times that the operation $\partial^2_\infty\mathsf{Sq}^2\mathrm{pr}^\infty_2$ has order $2$. 
    Hence, $\alpha_1=\partial^2_\infty\mathsf{Sq}^2\mathrm{pr}^\infty_2$ and $\beta_2 = \partial^2_\infty\mathsf{Sq}^2\mathrm{pr}^\infty_2$, while all other components are trivial.
\end{proof}
Under suitable efforts, these computations could be extended. 
Nevertheless, they suffice to deduce the slice $\dd_1$-differentials for special linear algebraic cobordism completely in the same range. The computation depends on the choice of the identification of $\slice_3(\MSL_\Lambda)$.
\begin{remark}\label{rem:choice_s3}
    According to Theorem \ref{slices_MSL}, there exists an isomorphism $\slice_3(\MSL_\Lambda)\cong \Sigma^{(3)}\EM\Lambda\oplus \Sigma^{2+(3)}\EM\Lambda\oplus \Sigma^{3+(3)}\EM\Lambda$. 
    However, as mentioned in part (2) of Theorem \ref{slices_MSL} and Remark \ref{rem:splitting_non_unique}, the map $\slice_3(\MSL_\Lambda)\to \slice_3(\MWL_\Lambda)$ depends on the choice of such an isomorphism. 
    We now fix a specific choice that is used in the computation of the first slice differentials. 
    The generator of $\cyc_6(\W)$ is given by $x_3$, by the discussion just after \cite[Theorem 7.1]{CLP} and the homomorphism $\iota_4^*(\delta_6)\colon \W_6\to \cyc_4(\W)$ sends $x_1^{*3}$ to $-2x_1^{*2}$ and $x_3$ to zero (see \cite[Lemma 6.5 and Theorem~6.10]{CLP} and recall that $\partial=-\delta$). 
    It follows that the map $\slice_3(\MSL_\Lambda)\to \slice_3(\MWL_\Lambda)$ has the form
    \[ \begin{pmatrix}
        0 & \partial^2_\infty & 0 \\ 0 & \phi & 1
    \end{pmatrix}\colon\Sigma^{(3)}\EM\Lambda/2\oplus\Sigma^{2+(3)}\EM\Lambda/2\oplus\Sigma^{3+(3)}\EM\Lambda \to \Sigma^{3+(3)}\EM\Lambda\{x_1^{*3}\}\oplus\Sigma^{3+(3)}\EM\Lambda\{x_3\}, \]
    for some map $\phi\in \{0,\partial^2_\infty\}$. 
    These two possibilities differ by the automorphism of the source given by the following matrix
    \[
    \begin{pmatrix}
        1 & 0 & 0 \\ 0 & 1 & 0 \\ 0 & \partial^2_\infty & 1
    \end{pmatrix}.
    \]
    We choose the identification of $\slice_3(\MSL_\Lambda)$ corresponding to the case $\phi=0$.
\end{remark}
\begin{theorem}\label{thm:d1_msl}
    For $0\leq q\leq 3$ the differential $\dd_1(q)\colon\slice_q(\MSL_\Lambda)\to \Sigma\slice_{q+1}(\MSL_\Lambda)$ is given by:
    \allowdisplaybreaks
    \begin{align*}
        & \dd_1(0)=\mathsf{Sq}^2 \mathrm{pr}^\infty_2\colon\EM\Lambda\to\EM\Lambda/2\to \Sigma^{1+(1)}\EM\Lambda/2, \\
        & \dd_1(1)=\begin{pmatrix}
        \mathsf{Sq}^2 \\
        \partial^2_\infty \mathsf{Sq}^2\mathsf{Sq}^1\end{pmatrix}\colon\Sigma^{(1)}\EM\Lambda/2\to \Sigma^{1+(2)}\EM\Lambda/2\oplus \Sigma^{3+(2)}\EM\Lambda, \\
        & \dd_1(2)=\begin{pmatrix} \mathsf{Sq}^2 & \tau \mathrm{pr}^\infty_2 \\ \mathsf{Sq}^3\mathsf{Sq}^1 & \mathsf{Sq}^2\mathrm{pr}^\infty_2 \\ 0 & 0 \end{pmatrix}\colon \Sigma^{(2)}\EM\Lambda/2\oplus\Sigma^{2+(2)}\EM\Lambda\to \Sigma^{1+(3)}\EM\Lambda/2\oplus \Sigma^{3+(3)}\EM\Lambda/2\oplus\Sigma^{4+(3)}\EM\Lambda, \\
        & \dd_1(3)=\begin{pmatrix}
            \mathsf{Sq}^2 & \tau & 0 \\ \mathsf{Sq}^3\mathsf{Sq}^1 & \mathsf{Sq}^2+\rho\mathsf{Sq}^1 & 0 \\ 0 & \begin{pmatrix}
            \partial^2_\infty\mathsf{Sq}^2\mathsf{Sq}^1 \\ 0
            \end{pmatrix} & \begin{pmatrix}
            0 \\ \partial^2_\infty\mathsf{Sq}^2\mathrm{pr}^\infty_2
            \end{pmatrix}\end{pmatrix} \colon\Sigma^{(3)}\EM\Lambda/2\oplus\Sigma^{2+(3)}\EM\Lambda/2\oplus\Sigma^{3+(3)}\EM\Lambda\to \\
            & \hspace{9cm}\to\Sigma^{1+(4)}\EM\Lambda/2\oplus\Sigma^{3+(4)}\EM\Lambda/2\oplus\Sigma^{5+(4)}\EM\Lambda^2.
    \end{align*}
    Here, in formulas for $\dd_1(2)$ and $\dd_1(3)$, we use the identification of $\slice_3(\MSL_\Lambda)$ specified in Remark \ref{rem:choice_s3}.
    For $q\geq 4$ the differential $\dd_1(q)$ restricts to the direct summand $\Sigma^{(q)}\EM\Lambda/2\oplus \Sigma^{2+(q)}\EM\Lambda/2$ by \[\begin{pmatrix} \mathsf{Sq}^2 & \tau \\ \mathsf{Sq}^3\mathsf{Sq}^1 & \mathsf{Sq}^2+\rho\mathsf{Sq}^1 \end{pmatrix}\colon\Sigma^{(q)}\EM\Lambda/2\oplus \Sigma^{2+(q)}\EM\Lambda/2\to \Sigma^{1+(q+1)}\EM\Lambda/2\oplus\Sigma^{3+(q+1)}\EM\Lambda/2.\]
\end{theorem}
\begin{proof}
    The formulas for $\dd_1^{\MSL_\Lambda}(0)$ and $\dd_1^{\MSL_\Lambda}(1)$ follow from \cite[Theorem 3.5]{very_eff_KQ} since $\slice_q(\MSL_\Lambda)\cong\slice_q(\kq_\Lambda)$ for $q\leq 2$ by Proposition \ref{prop:slices-msl-kq}. 
    Moreover, it also gives the first $2\times 2$ block of $\dd_1^{\MSL_\Lambda}(q)$ for $q\geq 2,$ because $\slice_q(\MSL_\Lambda)\to\slice_q(\kq_\Lambda)$ maps the first two torsion summands identically by Proposition~\ref{prop:slices-msl-kq}. 
    Thus, it remains to compute the last row of $\dd_1^{\MSL_\Lambda}(2)$ and the last row and column of $\dd_1^{\MSL_\Lambda}(3)$.

    For $q=2$, the component of the first slice differential that goes from $\Sigma^{(2)}\EM\Lambda/2$ to the integral summand is zero by degree reasons \cite[Lemma A.4]{RO16}. 
    To prove the triviality of $\dd_1$ within the integral summands, consider the following commutative diagram (in which all terms are desuspended by $\Sigma^{(2)}$ for simplicity):
    \begin{center}\begin{tikzcd} \EM\Lambda/2\oplus\Sigma^{2}\EM\Lambda\ar[r] \ar[d,"{\dd_1^{\MSL_\Lambda}(2)}"'] & \Sigma^{2}\EM\Lambda\{x_1^{*2}\} \ar[d,"{\dd_1^{\MWL_\Lambda}(2)=(\begin{smallmatrix}
        \partial^2_\infty\mathsf{Sq}^2\mathrm{pr}^\infty_2 \\ 0 
    \end{smallmatrix})}"] \\ \Sigma^{1+(1)}\EM\Lambda/2\oplus\Sigma^{3+(1)}\EM\Lambda/2\oplus\Sigma^{4+(1)}\EM\Lambda \ar[r] & \Sigma^{4+(1)}\EM\Lambda\{x_1^{*3}\}\oplus\Sigma^{4+(1)}\EM\Lambda\{x_3\}
    \end{tikzcd}\end{center}
    This diagram is induced by the map $\MSL_\Lambda\to\MWL_\Lambda$. 
    The generator of $\cyc_4(\W)$ is given by $x_1^{*2}$ (i.e., the top morphism is the projection onto the second direct summand) and the bottom morphism is discussed in Remark \ref{rem:choice_s3}; recall that we use the identification of $\slice_3(\MSL_\Lambda)$ corresponding to $\phi=0$. 
    Now composing the above diagram with the projection onto $\Sigma^{4+(1)}\EM\Lambda\{x_3\}$, the identification of $\dd_1^{\MSL_\Lambda}(2)$ follows. 

    Our argument for $q=3$ proceeds similarly. 
    The components 
    \[\Sigma^{(3)}\EM\Lambda/2\to\Sigma^{5+(4)}\EM\Lambda^2\quad \text{and}\quad \Sigma^{3+(3)}\EM\Lambda\to \Sigma^{1+(4)}\EM\Lambda/2\] 
    of $\dd_1^{\MSL_\Lambda}(3)$ are again trivial by degree reasons. 
    The restriction of $\dd_1^{\MSL_\Lambda}(3)$ to $\Sigma^{3+(3)}\EM\Lambda\to\Sigma^{3+(4)}\EM\Lambda/2$ is zero, because the integral summand maps trivially under $\slice_3(\MSL_\Lambda)\to\slice_3(\kq_\Lambda)$, while mod $2$ summands maps identically onto the one from $\slice_4(\kq_\Lambda)$ by Proposition~\ref{prop:slices-msl-kq}. 
    Hence it remains to compute the last two entries of the last row. 
    For that, consider the following implicitly $\Sigma^{(3)}$-desuspended commutative diagram:
    \begin{center}
    \begin{tikzcd}
        \EM\Lambda/2 
        \oplus \Sigma^{2}\EM\Lambda/2 
        \oplus \Sigma^{3}\EM\Lambda 
        \arrow[d] 
        \arrow[r, "\dd_1^{\mathrm{MSL}_\Lambda}(3)"] 
        & 
        \Sigma^{1+(1)}\EM\Lambda/2 
        \oplus \Sigma^{3+(1)}\EM\Lambda/2 
        \oplus \Sigma^{5+(1)}\EM\Lambda 
        \oplus \Sigma^{5+(1)}\EM\Lambda 
        \arrow[d] \\
        \Sigma^{3}\EM\Lambda\{x_1^{*3}\} 
        \oplus \Sigma^{3}\EM\Lambda\{x_3\} 
        \arrow[r, "\dd_1^{\mathrm{MWL}_\Lambda}(3)"] 
        & 
        \Sigma^{5+(1)}\EM\Lambda\{x_1^{*4}\} 
        \oplus \Sigma^{5+(1)}\EM\Lambda\{x_1 * x_3\} 
        \oplus \Sigma^{5+(1)}\EM\Lambda\{x_4\}.
    \end{tikzcd}
    \end{center}
    Note that the summand $\EM\Lambda/2$ in its upper left corner maps trivially to the bottom right corner for degree reasons.
    The chosen generators of $\cyc_8(\W)$ are given by $x_1^{*4}$ and $2x_4-x_1*x_3$. 
    The left vertical map in the diagram is computed above and the right vertical map is readily computed by part (2) of Theorem~\ref{slices_MSL} (note that there are no boundary morphisms here since $\hmlg_6(\W)=0$).
    %First, we consider the summand of the target that corresponds to $x_1^{*4}$. 
    By \cite[Lemma~A.4]{RO16} and Lemma \ref{lemma:d1-integralsummands}, there are operations $\alpha_1,\alpha_2\in \mathbb{F}_2\{\partial^2_\infty\Sq^2\Sq^1\}$ and $\beta_1,\beta_2\in\mathbb{F}_2\{\partial^2_\infty\Sq^2\pr^\infty_2\}$ such that the right bottom $2\times 2$ block of $d_1^{\MSL_\Lambda}(3)$ has the form
    \[ \begin{pmatrix}
    \alpha_1 & \beta_1 \\ \alpha_2 & \beta_2 \end{pmatrix}\colon\Sigma^{2}\EM\Lambda/2\oplus\Sigma^{3}\EM\Lambda\to \Sigma^{5+(1)}\EM\Lambda\{x_1^{*4}\}\oplus \Sigma^{5+(1)}\EM\Lambda\{2x_4-x_1*x_3\}.\]
    Applying the above diagram, we obtain $\alpha_1=\partial^2_\infty\Sq^2\Sq^1$ and $\beta_2=\partial^2_\infty\Sq^2\pr^\infty_2$, while $\alpha_2$ and $\beta_1$ are trivial.
\end{proof}
\begin{remark}
    It seems possible to deduce formulas for the first slice differential in a wider range. 
    However, doing so along the arguments given above requires a presentation of the polynomial generators $x_i$ as linear combinations of the monomial generators $a_I$, which involves non-trivial combinatorial computations with Chern numbers.
\end{remark}
\begin{lemma}\label{lemma:E^2_for_m<3}
    The morphism $\MSL\to\kq$ induces isomorphisms \[E^2_{-n+m,q,-n}(\MSL_\Lambda)\xrightarrow{\simeq} E^2_{-n+m,q,-n}(\kq_\Lambda)\] for $m\leq 2$. 
    The description of these groups is given in the following table
    \begin{center}\begin{tabular}{|c|c|c|c|}
        \hline
         $q$ & $E^2_{-n,q,-n}$ & $E^2_{-n+1,q,-n}$ & $E^2_{-n+2,q,-n}$ \\
         \hline
         $0$ & $H^{n,n}$ & $H^{n-1,n}$ & $\kernel(\mathsf{Sq}^2\mathrm{pr}^\infty_2)\subseteq H^{n-2,n}$ \\
         $1$ & $h^{n+1,n+1}$ & $h^{n,n+1}/\mathsf{Sq}^2\mathrm{pr}^\infty_2H^{n-2,n}$ & $\kernel(\mathsf{Sq}^2)\subseteq h^{n-1,n+1}$ \\
         $2$ & $h^{n+2,n+2}$ & $h^{n+1,n+2}/\mathsf{Sq}^2h^{n-1,n+1}$ & $h^{n,n+2}/(\mathsf{Sq}^2h^{n-2,n+1})\oplus 2H^{n+2,n+2}$ \\
         $3$ & $h^{n+3,n+3}$ & $0$ & $h^{n+1,n+3}/(\mathsf{Sq}^2h^{n-1,n+2}+\tau\mathrm{pr}^\infty_2H^{n+1,n+2})$ \\
         $\geq 4$ & $h^{n+q,n+q}$ & $0$ & $0$\\
         \hline
    \end{tabular}\end{center}
\end{lemma}
\begin{proof}
    This follows directly from Theorem \ref{thm:d1_msl}, Proposition \ref{prop:slices-msl-kq}, and \cite[Theorem 2.6]{RSO24}.
\end{proof}
\begin{proposition}\label{prop:E^2_for_m=3}
    The non-trivial entries in the 3rd column of the $E^2$-page of the $-n$th slice spectral sequence for $\MSL_\Lambda$ are given as follows:
    \begin{align*}
        & E^2_{-n+3,0,-n}(\MSL_\Lambda)\cong H^{n-3,n} \\
        & E^2_{-n+3,1,-n}(\MSL_\Lambda)\cong\kernel(\mathsf{Sq}^2)\subseteq h^{n-2,n+1} \\
        & E^2_{-n+3,2,-n}(\MSL_\Lambda)\cong \kernel((b,B)\mapsto(\tau\mathrm{pr}^\infty_2(B)+\mathsf{Sq}^2(b),\mathsf{Sq}^3\mathsf{Sq}^1(b))\subseteq h^{n-1,n+2}\oplus H^{n+1,n+2} \\
        & E^2_{-n+3,3,-n}(\MSL_\Lambda)\cong h^{n,n+2}/(\mathrm{pr}^\infty_2H^{n,n+2})\oplus H^{n+3,n+3}\cong {}_2H^{n+1,n+2}\oplus H^{n+3,n+3}
    \end{align*}
    Moreover, the morphism between 3rd columns $E^2_{-n+3,q,-n}(\MSL_\Lambda)\to E^2_{-n+3,q,-n}(\kq_{\Lambda})$ is an isomorphism for $q\neq 3$ and a split epimorphism for $q=3$, namely the projection onto the first direct summand.
\end{proposition}
\begin{proof}
    For $q\neq 3$, the result follows immediately from the computation for $\kq$ \cite[\S 2]{RSO24}, Theorem \ref{thm:d1_msl} and Proposition \ref{prop:slices-msl-kq}. 
    In the remaining slice degree, the computation of the $\dd_1$-differential in Theorem~\ref{thm:d1_msl} implies that the integral motivic cohomology group $H^{n+3,n+3}$ is a direct summand of $E^2_{-n+3,3,-n}(\MSL_\Lambda)$. 
    Here we use the fact that the operation $\partial^2_\infty\mathsf{Sq}^2\mathrm{pr}^\infty_2$, which appears in the differential $\slice_3\to\Sigma\slice_4$, acts trivially on $H^{n+3,n+3}$. 
    Comparing with the slices of $\kq_\Lambda$ again, the remaining piece is seen to coincide with $E^2_{-n+3,3,-n}(\kq_\Lambda)$; the operation $\partial^2_\infty\mathsf{Sq}^2\mathsf{Sq}^1$ does not affect the summand $h^{n+2,n+3}$ since its target is zero.
\end{proof}
To determine the abutment in the 3rd column, we also need to identify some non-trivial entries in the 4th column of the $E^2$-page.
\begin{lemma}
    The group $E^2_{-n+4,0,-n}(\MSL_\Lambda)$ is isomorphic to $H^{n-4,n}$. 
    The morphism $\MSL\to\kq$ induces an isomorphism \[E^2_{-n+4,0,-n}(\MSL_\Lambda)\xrightarrow{\simeq} E^2_{-n+4,0,-n}(\kq_\Lambda).\]
\end{lemma}
\begin{proof}
    By Theorem \ref{thm:d1_msl}, the desired group is given by the kernel of the map \[\mathsf{Sq}^2\mathrm{pr}^\infty_2\colon H^{n-4,n}\to h^{n-2,n+1},\]
    which is trivial since $\Sq^2\colon h^{n-4,n}\to h^{n-2,n+1}$ is zero \cite[Corollary 6.2]{RO16}.
\end{proof}
\begin{lemma}\label{lemma:E2_4-1}
    The group $E^2_{-n+4,1,-n}(\MSL_\Lambda)$ is isomorphic to $h^{n-3,n+1}$. 
    The morphism $\MSL\to\kq$ induces an isomorphism \[E^2_{-n+4,1,-n}(\MSL_\Lambda)\xrightarrow{\simeq} E^2_{-n+4,1,-n}(\kq_\Lambda).\]
\end{lemma}
\begin{proof}
    Both $\dd_1$-differentials 
    \[\mathsf{Sq}^2\mathrm{pr}^\infty_2\colon H^{n-5,n}\to h^{n-3,n+1}\quad \text{and}\quad \begin{pmatrix}
        \mathsf{Sq}^2 \\ \partial^2_\infty\mathsf{Sq}^2\mathsf{Sq}^1
    \end{pmatrix}\colon h^{n-3,n+1}\to h^{n-1,n+2}\oplus H^{n+1,n+2}\]
    are trivial by \cite[Corollary 6.2]{RO16} again.
\end{proof}
\begin{lemma}\label{lemma:E2_4>4}
    For $q>4$ the group $E^2_{-n+4,q,-n}(\MSL_\Lambda)$ is isomorphic to $h^{n+q,n+q}$, and the morphism $\MSL\to \kq$ induces an isomorphism 
    \[E^2_{-n+4,q,-n}(\MSL_\Lambda)\xrightarrow{\simeq} E^2_{-n+4,q,-n}(\kq_\Lambda).\]
\end{lemma}
\begin{proof}
    According to the respective computation of slices and Proposition \ref{prop:slices-msl-kq}, the map $\MSL\to \kq$ induces isomorphisms 
    \[ E^1_{-n+4,q,-n}(\MSL_\Lambda)\xrightarrow{\simeq} E^1_{-n+4,q,-n}(\kq_\Lambda)\]
    for $q>4$. 
    Hence, the differential $\dd_1\colon E^1_{-n+4,q,-n}\to E^1_{-n+3,q+1,-n}$ has the form 
    \[ \begin{pmatrix}
        \mathsf{Sq}^2 & \tau & 0 \\ \mathsf{Sq}^3\mathsf{Sq}^1 & \mathsf{Sq}^2+\rho\mathsf{Sq}^1 & 0 
    \end{pmatrix}\colon h^{n+q-4,n+q}\oplus h^{n+q-2,n+q}\oplus h^{n+q,n+q} \to h^{n+q-2,n+q+1}\oplus h^{n+q,n+q+1} \]
    by \cite[Theorem 3.5]{very_eff_KQ}. 
    The first column acts trivially in this degree by \cite[Corollary 6.2]{RO16}. 
    Since multiplication with $\tau$ is injective, the kernel of the above matrix is then given by $h^{n+q-4,n+q}\oplus h^{n+q,n+q}$. 
    Now for $q>5$, the differential $$\dd_1\colon E^1_{-n+5,q-1,-n}\to E^1_{-n+4,q,-n}$$ has the form
    \begin{align*}
        \begin{pmatrix}
        \mathsf{Sq}^2 & \tau & 0 \\ \mathsf{Sq}^3\mathsf{Sq}^1 & \mathsf{Sq}^2+\rho\mathsf{Sq}^1 & 0 \\ 0 & \mathsf{Sq}^3\mathsf{Sq}^1 & \mathsf{Sq}^2 
    \end{pmatrix}\colon h^{n+q-6,n+q-1}\oplus h^{n+q-4,n+q-1}&\oplus h^{n+q-2,n+q-1}\to \\ &h^{n+q-4,n+q}\oplus h^{n+q-2,n+q}\oplus h^{n+q,n+q}
    \end{align*}
    by the same argument. 
    The operations $\mathsf{Sq}^2$ in the first and third rows act trivial in the given degree. 
    It follows that the group $E^2_{-n+4,q,-n}$ is isomorphic to the cokernel of
    \[ (\tau\ \mathsf{Sq}^3\mathsf{Sq}^1)\colon h^{n+q-4,n+q-1}\to h^{n+q-4,n+q}\oplus h^{n+q,n+q}, \]
    which can be identified with $h^{n+q,n+q}$ via the composition 
    \[ h^{n+q,n+q}\hookrightarrow h^{n+q-4,n+q}\oplus h^{n+q,n+q}\twoheadrightarrow \coker(\tau\ \mathsf{Sq}^3\mathsf{Sq}^1). \]
    In the remaining case $q=5$, the differential $\dd_1\colon E^1_{-n+5,4,-n}\to E^1_{-n+4,5,-n}$ has the form
    \begin{align*}
        \begin{pmatrix}
        \mathsf{Sq}^2 & \tau & 0 \\ \mathsf{Sq}^3\mathsf{Sq}^1 & \mathsf{Sq}^2+\rho\mathsf{Sq}^1 & \star \\ 0 & \mathsf{Sq}^3\mathsf{Sq}^1 & 0 \end{pmatrix} \colon h^{n-1,n+4}\oplus h^{n+1,n+4}\oplus (H^{n+3,n+4})^2\to h^{n+1,n+5}\oplus h^{n+3,n+5}\oplus h^{n+5,n+5}.
    \end{align*}
    Here the component $\EM\Lambda^2\to\Sigma^{1+(1)}\EM\Lambda/2$ of $\dd_1$ might be non-zero, but it acts trivially on $(H^{n+3,n+4})^2$ since the Steenrod square $\mathsf{Sq}^2\colon h^{n+3,n+4}\to h^{n+5,n+5}$ is zero. 
    Hence, the desired group $E^2_{-n+4,5,-n}$ can be computed in the same way as before. 
    The result follows.
\end{proof}

\section{Higher slice differentials for motivic sphere and \texorpdfstring{$\SL$}{SL}-cobordism}\label{section:5}
In Section~\ref{section:4} we identified the part of the second page of the slice spectral sequence for $\MSL$ which is responsible for the homotopy groups $\pi_{n+m,n}\MSL$, where $n\in \Z$ and $m\in \{0,1,2,3\}$.
To turn the pages further, information on higher slice differentials is required.
We start with the differentials entering the first three non-trivial columns.
\begin{lemma}\label{lem:higherdiffMSL}
    The slice $\dd_r$-differentials $E^r_{-n+m,q,-n}(\MSL_\Lambda)\to E^{r}_{-n+m-1,q+r,-n}(\MSL_\Lambda)$ are trivial for $m\leq 3$ and $r\geq 2$.
\end{lemma}
\begin{proof}
    This follows from Lemma \ref{lemma:E^2_for_m<3} and triviality of the slice $\dd_r$-differentials for $\kq$, see \cite[\S 2]{RSO24}. 
    Note that any $\dd_r$-differential with source $E^r_{-n+3,3,-n}(\MSL_\Lambda)$ is automatically trivial since its target is zero.
\end{proof}
\begin{lemma}\label{lemma:stab_m<3}
    For $m\leq 2$, the abutment $E^2_{-n+m,q,-n}(\MSL_\Lambda)=E^\infty_{-n+m,q,-n}(\MSL_\Lambda)$ occurs at the second page.
\end{lemma}
\begin{proof}
    This follows immediately from Lemma~\ref{lem:higherdiffMSL}.
\end{proof}
In order to compute the third Milnor--Witt stem of $\MSL_\Lambda$, we need to control the $3$rd column of the slice spectral sequence. 
Potentially non-trivial slice $\dd_r$-differentials that may affect it are
\begin{align*}
    \dd_2\colon E^2_{-n+4,0,-n}(\MSL_\Lambda)\to E^2_{-n+3,2,-n}(\MSL_\Lambda), \\
    \dd_2\colon E^2_{-n+4,1,-n}(\MSL_\Lambda)\to E^2_{-n+3,3,-n}(\MSL_\Lambda), \\
    \dd_3\colon E^3_{-n+4,0,-n}(\MSL_\Lambda)\to E^3_{-n+3,3,-n}(\MSL_\Lambda).
\end{align*}
To analyze these, recall that the $k$-th Milnor--Witt stem \[\pi_k\EE=\bigoplus_{n\in \Z}\pi_{n+k,n}\EE=\bigoplus_{n\in \Z}\pi_{k+(n)}\EE\] of a motivic spectrum $\EE\in \SHF$ is naturally a graded module over the zeroth Milnor--Witt stem $\pi_0\sph_F$ of the sphere spectrum over $F$.
Morel's Theorem identifies $\pi_0\sph_F$ with the Milnor--Witt $K$-theory of $F$ \cite{Morel_book}.
The latter is denoted $\KMW_\star(F)$.
Let $\KMil_\star(F)$ denote the Milnor $K$-theory of $F$ as a graded module over $\KMW_\star(F)$, and let $\KMiltwo_\star(F)$ be its reduction modulo $2$. 
%If $A$ is a graded $\KMW_\star(F)$-module and $n$ is an integer, then $A(n)$ is the graded $\KMW_\star(F)$-module with $A_m(n):=A_{n+m}$.

\begin{lemma}\label{lem:d2pi4s0sphere}
The differential $\dd_2\colon E^2_{-n+4,0,-n}(\sph_\Lambda)\to E^2_{-n+3,2,-n}(\sph_\Lambda)$ is nonzero in general.
More precisely, it is uniquely determined by fitting into a commutative diagram
\begin{center}\begin{tikzcd}
    H^{n-4,n}\ar[r,"\pr^\infty_{24}"] \ar[d,"\dd_2(\sph_\Lambda)"] &  h^{n-4,n}_{24} \ar[d,"\dd_2(\sph_\Lambda/12\hyper)"] \\  \ker(\Sq^2\colon h^{n-1,n+2}\to h^{n+1,n+3})\directsum h^{n,n+2}_{12}\ar[r] &      E^2_{-n+3,3,-n}(\sph_\Lambda/12\hyper)
\end{tikzcd}\end{center}
where the lower horizontal map is a split injection and the homomorphism $\dd_2(\sph_\Lambda/12\hyper)$ is determined by sending a generator $g\in h^{0,4}_{24}$ to $\widetilde{\tau^2\rho^4}\in h^{4,6}_{12}$.
This element is uniquely defined by the property $\pr^{12}_2(\widetilde{\tau^2\rho^4})=\tau^2\rho^4\in h^{4,6}$ over prime fields.
\end{lemma}

\begin{proof}
    For simpler notation, the coefficient ring $\Lambda$ will not be mentioned.
    The target of this differential is the direct sum of $\kernel (\mathsf{Sq}^2\colon h^{n-1,n+2}\to h^{n+1,n+3})$ and $h^{n,n+2}_{12}\cong h^{n,n+2}_4\oplus h^{n,n+4}_3$ by \cite[Lemma 3.10]{RSO24}.
  Consider the canonical map $\sph\to \sph/{12\hyper}$ to the Moore spectrum for the hyperbolic form of rank $24$. 
  The first differentials on $\slice_q(\sph/{12\hyper})$ for $q\leq 3$ are given in \cite[Theorem 6.6]{roendigs.moore}.
  The first differential $\dd_1\colon \pi_4\slice_0(\sph/12\hyper)\to \pi_3\slice_1(\sph/12\hyper)$ is again zero, since $\Sq^2\pr^{24}_2$ and $\partial^{24}_2$ are zero on elements in $h^{n-4,n}_{24}$.
  The first differential $\dd_1\colon \pi_4\slice_1(\sph/12\hyper)\to \pi_3\slice_2(\sph/12\hyper)$ is zero on the summand $h^{n-3,n+1}$ since $\Sq^1$ and $\Sq^2$ are zero on such elements, but not necessarily on the summand $h^{n-2,n+1}$.
  Nevertheless, the image of this first differential intersects the image of
  $\pi_3\slice_2\sph\to\pi_3\slice_2(\sph/12\hyper)$ in the zero subgroup.
  It follows that the lower horizontal map in the diagram
  \begin{center}
    \begin{tikzcd}
      \pi_{4}\slice_0\sph\ar[r] \ar[d,"\dd_2"] &      \pi_{4}\slice_0(\sph/12\hyper) \ar[d,"\dd_2"]\\
      E^2_{\star+3,2,\star}(\sph) \ar[r] & E^2_{\star+3,2,\star}(\sph/12\hyper)
    \end{tikzcd}
  \end{center} 
  induced by $\sph\to \sph/12\hyper$ is injective.
  The $\KMW_\star(F)$-module $\pi_4\slice_0(\sph/12\hyper)$ is generated in $\pi_{4-(4)}\slice_0(\sph/12\hyper)$ by an element $g\in h^{0,4}_{24}$, which already lives over the prime field $F_0$ of $F$.
  Hence, its image under the second slice differential resides in the group
  \begin{align*} E^2_{-1,2,-4}(\sph/12\hyper) & = \ker(\Sq^2\colon h^{3,6}(F_0)\to h^{5,7}(F_0))
  \directsum h^{4,6}_{12}(F_0)\directsum \\ 
  & \directsum \{(b,B)\in h^{4,6}(F_0)\directsum h^{5,6}_{12}(F_0)\colon \Sq^2b=
  \tau\partial^{12}_2(B)\}/(\Sq^2,\inc^2_{12}\Sq^2\Sq^1)h^{2,5}(F_0).
  \end{align*}
  For prime fields of positive characteristic, this group is zero, whence the image of the second differential $\dd_2\colon \pi_4\slice_0\sph\to \pi_3\slice_2\sph$ must be zero.
  For the field of rational numbers, this group reduces to $h^{4,6}_{12}(\QQ)$, which coincides with the image of the lower horizontal injection in the diagram above for $\QQ$.
  This group is isomorphic via $\pr^{12}_2$ to $h^{4,6}(\QQ)\iso \Z/2\{\tau^2\rho^4\}$, because $\partial^2_6\colon h^{3,6}(\QQ)\to h^{4,6}_6(\QQ)$ is an isomorphism and $\partial^2_6\colon h^{4,6}(\QQ)\to h^{5,6}_6(\QQ)$ is the zero homomorphism.
  In particular, there are only two possibilities for the differential $\dd_2(\sph/12\hyper)$, sending any chosen generator either to the preimage $\widetilde{\tau^2\rho^4}$ of $\tau^2\rho^4$ or to $0$.
 
  To produce a contradiction, assume the chosen generator maps to $0$.
  Then the second slice differential in this degree is zero, both for $\sph/12\hyper$ and for $\sph$, over all fields. 
  In particular, this holds for $\RR$. 
  The field inclusion $\QQ\hookrightarrow \RR$ induces an isomorphism $h^{n,n+2}_{12}(\QQ)\iso h^{n,n+2}_{12}(\RR)$ for all $n>2$.
  Relevant will be the case $n=5$.
  The element $\widetilde{\tau^2\rho^5}\in h^{5,7}(\RR)$ is thus a nonzero permanent cycle, which cannot be hit by any differential, being in slice degree $2$.
  It thus represents a nonzero element $\xi\in \pi_{3-(5)}\sph^\wedge_\eta(\RR)$.
  Thanks to Lemma~\ref{lem:unit-d3pi4s5}, there are only finitely many nonzero groups from higher slices contributing to $\pi_{3-(5)}\sph^\wedge_\eta$, and these are all finite $2$-torsion groups for $\RR$.
  Hence there exists $r\in \NN$ such that the element $\xi$ is a nonzero $2^r$-torsion element in $\pi_{3-(5)}\sph^\wedge_\eta(\RR)$ which cannot come from slice filtration level $>2$.
  Theorem~\ref{thm:pi3-(5)reals} below implies that the only $2$-primary torsion in $\pi_{3-(5)}\sph_\RR$ is generated by $\sigma\in \pi_{3+(4)}\sph$.
  In particular, it comes from slice filtration level $4$.
  This is a contradiction.

  It follows that the chosen generator has to map to $\widetilde{\tau^2\rho^4}\in h^{4,6}_{12}$ over $\QQ$, and thus over any field. 
  (It is the zero element over fields of positive characteristic.)
  To see that this determines the second slice differential for $\sph$ in this degree, it remains to note that since its target is $24$-torsion (even $12$-torsion), it factors over
  $H^{n-4,n}/24H^{n-4,n}$, which embeds in the upper right corner $h^{n-4,n}_{24}$ in the diagram above.

  To prove the statement that this second slice differential can indeed be nonzero, it suffices to provide an element in $\pi_{4-(n)}\slice_0(\sph)$ for some field which maps nontrivially under $\dd_2$.
  Let $F=\QQ$ and $n=5$.
  The group $\pi_{4-(5)}\slice_0(\sph_\QQ)\iso H^{1,5}(\QQ)\iso H^{1,5}(\Z)\iso \Z\oplus \Z/2\Z$ is nontrivial, and the image of $H^{1,5}(\QQ)/24\iso \Z/24\Z\oplus \Z/2\Z\hookrightarrow h^{1,5}_{24}(\QQ)$ contains $\rho g$ for any generator $g\in h^{0,4}_{24}(\QQ)\iso h^{0,4}_{24}(\Z)\iso\Z/24\Z$.
  The reason is that $\partial^{24}_\infty(g)\in H^{1,4}(\QQ)\iso H^{1,4}(\Z)\iso \Z/240\Z$ lives already over $\Z$, just as $\rho$, whereas $\partial^{24}_\infty(\rho g)=\rho \partial^{24}_{\infty}(g)\in H^{2,5}(\Z)=0$, since $K_8(\Z)=0$ \cite{k8z}.
  In particular, there exists an element in $H^{1,5}(\QQ)$ which is mapped to the unique nonzero element $\widetilde{\tau^2\rho^5}\in h^{5,7}_{12}(\QQ)$.
\end{proof}

\begin{corollary}\label{cor:kq-msl-d2pi4s0}
  The second slice differential $\dd_2\colon E^2_{-n+4,0,-n}(\kq) \to E^2_{-n+3,2,-n}(\kq)$ is zero.
  The same holds for $\MSL_\Lambda$ in place of $\kq$.
\end{corollary}

\begin{proof}
    The unit $\sph\to \kq$ induces an equivalence on $\slice_0$, and the map
  \[ \slice_2\sph\iso \Sigma^{(2)}\EM\Z/2 \oplus \Sigma^{1+(2)}\EM\Z/12 
    \xrightarrow{\begin{pmatrix} 1 & 0 \\ 0 & \partial^{12}_\infty \end{pmatrix}}
    \Sigma^{(2)}\EM\Z/2 \oplus \Sigma^{2+(2)}\EM\Z\iso \slice_2\kq\]
  on $\slice_2$ after a suitable choice of equivalence.
  The image of $\dd_2\colon E^2_{-n+4,0,-n}(\sph_\Lambda) \to E^2_{-n+3,2,-n}(\sph_\Lambda)$
  is contained in the Milnor--Witt submodule of $h^{n,n+2}_{12}$ generated by
  $\widetilde{\tau^2\rho^4}\in h^{4,6}$ by Lemma~\ref{lem:d2pi4s0sphere}.
  This element, which lives over the prime field $F_0\subset F$, is sent to $\partial^{12}_{\infty}(\widetilde{\tau^2\rho^4})\in H^{5,6}(F_0)=0$.
  Hence the image of $\dd_2\colon E_{-n+4,0,-n}(\kq) \to E^2_{-n+3,2,-n}(\kq)$ must be zero for any field.
  The same then holds for $\MSL_\Lambda$, because the special linear orientation $\MSL_\Lambda\to \kq_\Lambda$ induces an equivalence on $\slice_q$ for $q\leq 2$.
\end{proof}

\begin{lemma}\label{lem:d2pi4s1sphere}
    The differential $\dd_2\colon E^2_{-n+4,1,-n}(\sph_\Lambda)\to E^2_{-n+3,3,-n}(\sph_\Lambda)$ is trivial.
    The same holds for $\kq$ and $\MSL_\Lambda$ in place of $\sph_\Lambda$.
\end{lemma}
\begin{proof}
    Again $\Lambda$ will be ignored in the notation.
    Applying \cite[Theorem 3.5]{very_eff_KQ} and \cite[Lemma~3.11]{RSO24}, this differential for $\kq$ has the form \[h^{*-3,*+1}\to h^{*,*+2}/(\mathrm{pr^\infty_2}H^{*,*+2}). \] 
    In the case of $\sph_\Lambda$, one can prove similarly that only the target is potentially different, namely of the form $h^{*,*+2}/\partial^{12}_{2}h^{*-1,*+2}_{12}$; see \cite[Lemma 3.11]{RSO24} for the corresponding statement one row below.
    Hence it would suffice to give the argument for $\sph_\Lambda$, but instead both cases are treated.
    The source in both cases is generated (as a module over Milnor--Witt $K$-theory) by the unique non-trivial element $g\in h^{0,4}\cong \Z/2$, which exists over the prime subfield $F_0\subset F$. 
    Its image under the differential for $\kq$ is contained in the group $h^{3,5}/(\mathrm{pr}^\infty_2H^{3,5})$. 
    We claim that this group is trivial over $F_0$. 
    If $F$ has odd characteristic, already $h^{3,5}(F_0)$ is trivial, since $\KMil_3(F_0)=0$. 
    In the case $F_0=\QQ$, the morphism $\mathrm{pr}^\infty_2\colon H^{3,5}\to h^{3,5}$ is surjective, because $\partial^2_\infty(\tau^3\rho^2)\in H^{3,5}$ hits the generator $\tau^2\rho^3\in h^{3,5}\cong \KMiltwo_3(\QQ)\cong\Z/2$.
    In the case of $\sph_\Lambda$, the argument is the same in positive characteristic. 
    For the rational numbers, the reason is that
    $\partial^{12}_{2}\colon h^{2,5}_{12}(\QQ)\to h^{3,5}(\QQ)\iso \Z/2\Z\{\tau^2\rho^3\}$
    is surjective.
    In fact, already $h^{2,5}(\QQ)\xrightarrow{\inc^{2}_{12}} h^{2,5}_{12}(\QQ)\xrightarrow{\partial^{12}_{2}} h^{3,5}(\QQ)$ is surjective.
    The case of $\MSL_\Lambda$ follows from $\sph_\Lambda$, because the unit $\sph\to\MSL$ induces an isomorphism 
    $E^2_{-n+4,1,-n}(\sph_\Lambda)\cong E^2_{-n+4,1,-n}(\MSL_\Lambda).$
\end{proof}

It remains to discuss a few third differentials.

\begin{lemma}\label{lem:unit-d3pi4s5}
  The differential $\dd_3\colon E^3_{-n+4,5,-n}(\sph_\Lambda)\to E^3_{-n+3,8,-n}(\sph_\Lambda)$ sends the unique nonzero element in $\pi_{4+(5)}\slice_5\sph_\Lambda\iso h^{0,0}\{\alpha_5\}$ to $\rho^3\in h^{3,3}\{\alpha_4\alpha_1^4\}$.
  More generally, the differential 
  \[ \dd_3\colon E^3_{-n+4,5+m,-n}(\sph_\Lambda)\to E^3_{-n+3,8+m,-n}(\sph_\Lambda)\] 
  sends the unique nonzero element in $\pi_{4+(5+m)}\slice_{5+m}\sph_\Lambda\iso h^{0,0}\{\alpha_5\alpha_1^{m}\}$ to $\rho^3\in h^{3,3}\{\alpha_4\alpha_1^{4+m}\}$ for every $m\in \NN$.
\end{lemma}

\begin{proof}
  This follows for $\sph[\eta^{-1}]$ from \cite{BH_eta}, which proved \cite[Conjecture 4.10]{ormsby-roendigs}.
  Using naturality with respect to the unit $\sph_\Lambda\to\sph_\Lambda[\eta^{-1}]$, this holds already for $\sph_\Lambda$.
\end{proof}

\begin{lemma}\label{lem:d3pi4s0kq}
The third slice differential
\[ \dd_3\colon E^3_{-n+4,0,-n}(\kq) \to E^3_{-n+3,3,-n}(\kq)\]
is zero.
\end{lemma}

\begin{proof}
The target of this third slice differential is the $2$-torsion group $h^{n,n+3}/\tau\pr^\infty_2H^{n,n+2}$, whereas the source is $H^{n-4,n}$, by Corollary~\ref{cor:kq-msl-d2pi4s0} and Lemma~\ref{lem:d2pi4s1sphere}.
While in principle the argument may proceed entirely within the slice filtration, the proof seems slightly easier using the very effective slice filtration on $\kq$.
Consider \cite{bachmann.hermitian} as a convenient reference for the very effective slice filtration over fields, dubbed ``generalized'' slice filtration.
It has the advantage that the $n$-th filtration $\vf_n\kq$ is $n$-connective, which in particular implies strong convergence for the associated spectral sequence.
The main result \cite[Theorem 16]{bachmann.hermitian}, generalized to include fields of characteristic 2 in \cite[Theorem 4.1]{KRO_hermitian_k_theory_char2}, provides the very effective slices $\vs_n\kq$ of $\kq$.
Since $\vs_3\kq\simeq \ast$, only $\pi_3\vs_n\kq$ for $n\in \{0,1,2\}$ contribute to $\pi_3\kq$.
The latter reference also identifies the first differentials $\vs_0\kq\to \Sigma \vs_1\kq \simeq \Sigma^{2+(1)}\EM\Z/2$ and $\vs_1\kq \to \Sigma\vs_2\kq\simeq \Sigma^{3+(2)}\EM\Z$ as the respective unique nonzero map; see \cite[Corollary 8.9]{KRO_hermitian_k_theory_char2}.
The resulting first page of the very effective slice spectral sequence for $\kq$ thus contains the contribution $\pi_{3-(n)}\vs_2\kq\cong H^{n+1,n+2}/\partial^2_\infty\Sq^2h^{n-2,n+1}$ to $\pi_3\kq$.
This group is hit by the second differential $d_2\colon \pi_{4-(n)}\vs_0\kq \to H^{n+1,n+2}/\partial^2_\infty\Sq^2h^{n-2,n+1}$ whose source fits into a short exact sequence
\[ 0 \to h^{n-3,n+1}\to \pi_{4-(n)}\vs_0\kq \to H^{n-4,n}\to 0 \]
by \cite[Prop.~7.6]{KRO_hermitian_k_theory_char2}; here \cite[Lemma 8.11]{KRO_hermitian_k_theory_char2} says that the first differential is zero on $\pi_4\vs_0\kq$.
The restriction of the second differential $d_2$ to $h^{\star-3,\star+1}$ is zero, because this Milnor-Witt module is generated by a unique element in $h^{0,4}$ over the prime field $F_0$.
Its image under the second differential $d_2$ thus lies in the group $H^{4,5}(F_0)/\partial^2_\infty \Sq^2 h^{1,4}(F_0)$ which is zero, since already $H^{4,5}(F_0)$ is.
Hence the second differential $d_2$ induces uniquely a homomorphism $H^{n-4,n}\to H^{n+1,n+2}/\partial^2_\infty\Sq^2h^{n-2,n+1}$.
The relation between the effective and the very effective slice filtration explained in \cite{bachmann.hermitian} factors it as
\[ H^{n-4,n}\xrightarrow{\dd_3} h^{n,n+3}/\tau\pr^\infty_2H^{n,n+2}\cong h^{n,n+2}/\pr^\infty_2H^{n,n+2}\xrightarrow{\gamma} H^{n+1,n+2}/\partial^2_\infty \Sq^2h^{n-2,n+1}\]
where the first homomorphism is the third slice differential, and $\gamma$ is induced by the connecting homomorphism $\partial^2_\infty \colon h^{n,n+2}\to H^{n+1,n+2}$.
Since $\pr^\infty_2H^{n,n+2}$ is precisely the kernel of the latter homomorphism, and moreover $\Sq^2h^{n-2,n+1}\subset \pr^2_\infty\partial^2_\infty h^{n-1,n+2}\subset \pr^2_\infty H^{n,n+2}$, the homomorphism $\gamma$ is injective.
Hence the third slice differential in question is zero if and only if the second very effective slice differential $d_2$ containing it as a composition factor is zero.

To limit the possible homomorphism that can occur as the second very effective slice differential in this degree, let $\mathsf{Q}$ be the fiber of the first differential $\vs_0\kq\to \Sigma\vs_1\kq$.
Choose a lift $\mathsf{Q} \to \Sigma \vf_2\kq$, which composes with the canonical map to $\ell \colon \mathsf{Q}\to \Sigma\vs_2\kq$.
There results a diagram
\begin{equation}\label{eq:d3-d2} \begin{tikzcd}
      \pi_4\mathsf{Q} \ar[r] \ar[d,"\ell"] & \pi_4\vs_0\kq \ar[d,"d_2"] \\
      \pi_3 \vs_2\kq \ar[r] & \pi_3\vs_2\kq/\partial^2_\infty\Sq^2 \pi_4\vs_1\kq
    \end{tikzcd} 
  \end{equation}
which commutes by the definition of the second slice differential.
Moreover, the upper horizontal arrow in~(\ref{eq:d3-d2}) depicts an epimorphism by \cite[Lemma 8.11]{KRO_hermitian_k_theory_char2}.
It remains to prove that for any possible map $\ell\colon \mathsf{Q}\to \Sigma\vs_2\kq$, the induced homomorphism $d_2$ is zero.
The group of maps $[\mathsf{Q},\Sigma\vs_2\kq]$ fits into an exact sequence
\begin{equation}\label{eq:second-veff-kq} [\Sigma\vs_1\kq,\Sigma\vs_2\kq] \to [\vs_0\kq, \Sigma\vs_2\kq] \to
    [\mathsf{Q},\Sigma\vs_2\kq] \to [\vs_1\kq,\Sigma\vs_2\kq] \to [\Sigma^{-1}\vs_0\kq,\Sigma\vs_2\kq]  
\end{equation}
in which the outer homomorphisms are induced by precomposition with the first very effective slice differential $\vs_0\kq \to \Sigma\vs_1\kq$.
In particular, the last homomorphism in~(\ref{eq:second-veff-kq}) is zero, because the only nontrivial element in its source is the first very effective slice differential $\vs_1\kq\to \Sigma\vs_2\kq$.
With the slice filtration on $\vs_0\kq$, one may determine $[\vs_0\kq,\Sigma\vs_2\kq]$ as an extension of $[\effcover_1\vs_0\kq,\Sigma\vs_2\kq]$ (which is isomorphic to $[\slice_1\vs_0\kq,\Sigma\vs_2\kq]\cong h^{0,0}\oplus h^{1,1}$, as the weight $1$ motivic Steenrod algebra implies \cite{Voevodsky_Steenrod, HKO_Steenrod}) by $[\slice_0\vs_0\kq,\Sigma\vs_2\kq]$ (which is isomorphic to $\Z/6$, as \cite[Lemma 8.5]{KRO_hermitian_k_theory_char2} implies).
The source of the first homomorphism in~(\ref{eq:second-veff-kq}) is isomorphic to $h^{1,1}$, and the first homomorphism hits the copy of $h^{1,1}$ in its target by direct inspection.
As a consequence, the abelian group $[\mathsf{Q},\Sigma\vs_2\kq]$ is an extension of $\Z/2$ by a group which is an extension of $\Z/6$ and $\Z/2$.
To prove that every element $\ell \in [\mathsf{Q},\Sigma\vs_2\kq]$ induces the zero homomorphism $d_2$ in the diagram~(\ref{eq:d3-d2}), consider first a generator
\[ \mathsf{Q}\to\vs_0\kq\to\slice_0\vs_0\kq=\slice_0\kq \xrightarrow{\partial^2_\infty\Sq^4\pr_2^\infty+\partial^3_\infty\mathsf{P}^1\pr^\infty_3} \Sigma\vs_2\kq \]
of $\Z/6$.
Having already seen that any second differential $d_2$ is zero on $h^{\star-3,\star+1} \subset \pi_4\vs_0\kq$, it suffices to prove that for every $x\in H^{n-4,n}$, the elements $\partial^2_\infty\Sq^4\pr^\infty_2(x)$ and $\partial^3_\infty\mathsf{P}^1\pr^\infty_3(x)$ are zero.
This holds, because $h^{\star-4,\star}$ and $h^{\star-4,\star}_3$ are generated by an element in degree $(0,4)$ as Milnor-Witt modules, respectively, which lives over the prime field $F_0$.
The target group $H^{5,6}(F_0)$ for both these elements is the zero group.
Any multiple of the generator above then also induces the zero homomorphism $d_2$.
Changing such a multiple by any element in one of the copies of $\Z/2$ occurring in the extensions described above does not change $d_2$ either.
For the first copy, it follows from the vanishing on $h^{\star-3,\star+1} \subset \pi_4\vs_0\kq$, while the other copy is given by maps from $\vs_1\kq$.
However, $\vs_1\kq$ maps trivially to $\vs_0\kq$ anyhow via $\mathsf{Q}$.
This concludes the proof.
\end{proof}

\begin{corollary}\label{cor:d3pi4s0msl}
  The first component of the third differential
  \[\dd_3\colon E^{3}_{-n+4,0,-n}(\MSL_\Lambda)=\pi_{4-(n)}\slice_0\MSL_\Lambda\to
  E^{3}_{-n+3,3,-n}(\MSL_\Lambda)\]
  is zero.
\end{corollary}

\begin{proof}
  This follows from Lemma~\ref{lem:d3pi4s0kq}, because the map $\MSL\to \kq$ induces the projection onto the first component 
  \[ h^{n,n+2}/(\mathrm{pr^\infty_2}H^{n,n+2})\oplus H^{n+3,n+3}\iso E^{3}_{-n+3,3,-n}(\MSL_\Lambda)\twoheadrightarrow E^{3}_{-n+3,3,-n}(\kq_\Lambda)\iso h^{n,n+2}/(\mathrm{pr^\infty_2}H^{n,n+2}) \]
  on the third pages of the respective slice spectral sequences.
\end{proof}

In the computation of the second component of the $\dd_3$-differential, we will reduce to a map of motivic Eilenberg--MacLane spectra. The following lemma will be useful.
\begin{lemma}\label{lem:MZtoMZ(3)[7]}
    Let $F$ be a field of characteristic not $2$.
    The group of maps of the form $\EM\Lambda\to\Sigma^{4+(3)}\EM\Lambda$ is isomorphic to 
    \[ \mathbb{F}_2\{\partial^2_\infty\Sq^6\pr^\infty_2\}\oplus\mathbb{F}_2\{\partial^2_\infty\Sq^4\Sq^2\pr^\infty_2\}. \]
\end{lemma}
\begin{proof}
The cofiber sequence $\effcover_4\sph_\Lambda\to \effcover_3\sph_\Lambda \to \slice_3\sph_\Lambda$ induces an isomorphism $[\effcover_3\sph_\Lambda,\Sigma^{s+(3)}\EM\Lambda]\cong [\slice_3\sph_\Lambda,\Sigma^{s+(3)}\EM\Lambda]$ for all integers $s$, because $\Sigma^{s+(3)}\EM\Lambda$ is right orthogonal to $4$-effective motivic spectra.
Since $\slice_3\sph_\Lambda\simeq \Sigma^{(3)}\EM\Lambda/2\oplus \Sigma^{2+(3)}\EM\Lambda/2$ by \cite[Corollary 2.13]{RSO19}, there results an isomorphism 
\[ [\effcover_3\sph_\Lambda,\Sigma^{s+(3)}\EM\Lambda]\cong \begin{cases} \mathbb{F}_2 & s\in \{1,3\} \\ 0 & \mathrm{else}\end{cases}\]
where for $s\in \{1,3\}$, the generator corresponds to the connecting map $\partial^2_\infty$.
The cofiber sequence $\effcover_3\sph_\Lambda\to \effcover_2\sph_\Lambda \to \slice_2\sph_\Lambda$ induces a not so long exact sequence
\[ 0 \to [\slice_2\sph_\Lambda,\Sigma^{3+(3)}\EM\Lambda] \to [\effcover_2\sph_\Lambda,\Sigma^{3+(3)}\EM\Lambda] \to \mathbb{F}_2 \to [\slice_2\sph_\Lambda,\Sigma^{4+(3)}\EM\Lambda] \] 
in which the last displayed homomorphism is induced by precomposition with the first slice differential $\Sigma^{-1}\slice_2\sph_\Lambda\simeq \Sigma^{-1+(2)}\EM\Lambda/2\oplus \Sigma^{(2)}\EM\Lambda/12\xrightarrow{\Sq^3\Sq^1+\Sq^2\partial^{12}_{2}} \Sigma^{2+(3)}\EM\Lambda/2$; see \cite[Lemma 4.1]{RSO19}.
Since $\partial^2_{\infty}\circ (\Sq^3\Sq^1+\Sq^2\partial^{12}_{2})= \partial^2_\infty\Sq^2\partial^{12}_2\neq 0$, this last homomorphism is injective, giving an isomorphism $[\slice_2\sph_\Lambda,\Sigma^{3+(3)}\EM\Lambda]\cong [\effcover_2\sph_\Lambda,\Sigma^{3+(3)}\EM\Lambda]$.
The form of $\slice_2\sph_\Lambda$ already displayed above provides an isomorphism
$[\slice_2\sph_\Lambda,\Sigma^{3+(3)}\EM\Lambda]\cong \mathbb{F}_2\{\partial^2_\infty\Sq^2\Sq^1\}\oplus \mathbb{F}_2\{\partial^2_\infty\Sq^2\pr^{12}_{2}\}$ by knowledge on the weight~$1$ part of the Steenrod algebra \cite{Voevodsky_Steenrod, HKO_Steenrod}.
Similarly, precomposition with the first slice differential $\Sigma^{-1}\slice_2\sph_\Lambda\simeq \Sigma^{-1+(2)}\EM\Lambda/2\oplus \Sigma^{(2)}\EM\Lambda/12\xrightarrow{\Sq^2+\tau\partial^{12}_{2}} \Sigma^{(3)}\EM\Lambda/2$ from \cite[Lemma 4.1]{RSO19} produces an exact sequence
\[  \mathbb{F}_2\to [\slice_2\sph_\Lambda,\Sigma^{2+(3)}\EM\Lambda] \to [\effcover_2\sph_\Lambda,\Sigma^{2+(3)}\EM\Lambda]\to 0\]
in which the first homomorphism sends $\partial^2_\infty$ to $\partial^2_\infty\Sq^2+\rho\partial^{12}_\infty\neq 0$.
The cofiber sequence $\effcover_2\sph_\Lambda\to \effcover_1\sph_\Lambda \to \slice_1\sph_\Lambda$ induces an exact sequence
\[  [\slice_2\sph_\Lambda,\Sigma^{2+(3)}\EM\Lambda]/\mathbb{F}_2\to [\slice_1\sph_\Lambda,\Sigma^{3+(3)}\EM\Lambda] \to [\effcover_1\sph_\Lambda,\Sigma^{3+(3)}\EM\Lambda] \to \mathbb{F}_2\oplus \mathbb{F}_2 \to [\slice_1\sph_\Lambda,\Sigma^{4+(3)}\EM\Lambda] \] 
in which the first and the last displayed homomorphism are induced by precomposition with the first slice differential $\Sigma^{-1}\slice_1\sph_\Lambda\simeq \Sigma^{-1+(1)}\EM\Lambda/2\xrightarrow{(\Sq^2,\inc^{2}_{12}\Sq^2\Sq^1)} \Sigma^{(2)}\EM\Lambda\oplus \Sigma^{1+(2)}\EM\Lambda/12$; see again \cite[Lemma 4.1]{RSO19}.
The composition $\partial^2_\infty\Sq^2\pr^{12}_2\inc^2_{12}\Sq^2\Sq^1=0$, while the composition $\partial^2_\infty\Sq^2\Sq^1\Sq^2=\partial^2_\infty(\Sq^4\Sq^1+\Sq^1\Sq^4)=\partial^2_\infty\Sq^4\Sq^1\neq 0$ by a motivic Adem relation \cite[Theorem 10.2]{Voevodsky_Steenrod}, \cite[Theorem 5.1]{HKO_Steenrod}.
The previous exact sequence thus reduces to a shorter exact sequence
\[  [\slice_2\sph_\Lambda,\Sigma^{2+(3)}\EM\Lambda]/\mathbb{F}_2 \to [\slice_1\sph_\Lambda,\Sigma^{3+(3)}\EM\Lambda] \to [\effcover_1\sph_\Lambda,\Sigma^{3+(3)}\EM\Lambda] \to \mathbb{F}_2\{\partial^2_\infty\Sq^2\pr^{12}_{2}\} \to 0 \] 
in which $[\slice_1\sph_\Lambda,\Sigma^{3+(3)}\EM\Lambda]\cong h^{1,1}(F)\{\partial^2_\infty\Sq^2\Sq^1\}\oplus \mathbb{F}_2\{\partial^2_\infty\Sq^4\}$, as the weight~$2$ part of the Steenrod algebra implies \cite{Voevodsky_Steenrod, HKO_Steenrod}.
The first summand is in the image of the leftmost homomorphism, since for every square class $\phi\in h^{1,1}(F)$ with representing unit $\Phi\in H^{1,1}(F)$ one has $\partial^2_\infty\phi\Sq^2\Sq^1=\Phi\partial^{12}_\infty\inc^{12}_{2}\Sq^2\Sq^1$. 
Thus the previous exact sequence reduces to a short exact sequence
\[  0\to \mathbb{F}_2\{\partial^2_\infty\Sq^4\}\to [\effcover_1\sph_\Lambda,\Sigma^{3+(3)}\EM\Lambda] \to \mathbb{F}_2\{\partial^2_\infty\Sq^2\pr^{12}_{2}\} \to 0 \]
which soon will turn out to be split.
The cofiber sequence $\effcover_1\sph\to \sph \to \slice_0\sph = \EM\Z$ induces an isomorphism 
$[\effcover_1\sph_\Lambda,\Sigma^{3+(3)}\EM\Lambda]\cong [\EM\Lambda,\Sigma^{4+(3)}\EM\Lambda]$, because $[\sph,\Sigma^{s+(3)}\EM\Z]=0$ for $s>0$ over any field.
Again one resulting nonzero element is induced by precomposition with a first slice differential, now $\Sq^2\pr^\infty_2\colon \EM\Lambda\simeq \slice_0\sph_\Lambda\to \Sigma\slice_1\sph_\Lambda\simeq\Sigma^{1+(1)}\EM\Lambda/2$, which provides $\partial^2_\infty\Sq^4\Sq^2\pr^\infty_2\in [\EM\Lambda,\Sigma^{4+(3)}\EM\Lambda]$. 
Also $\partial^2_\infty\Sq^6\pr^\infty_2\in [\EM\Lambda,\Sigma^{4+(3)}\EM\Lambda]$, and it is nonzero, since for $F=\QQ$, the composition 
\[ H^{1,5}(\QQ)\to h^{1,5}(\QQ)\xrightarrow{\Sq^6} h^{7,8}(\QQ)\to H^{8,8}(\QQ)\] 
is nonzero by Lemma~\ref{lem:Sq6_tau4} and the argument appearing in the proof of Lemma~\ref{lem:d2pi4s0sphere}.
It also differs from $\partial^2_\infty\Sq^4\Sq^2\pr^\infty_2$, because the composition 
\[ H^{1,5}(\QQ)\to h^{1,5}(\QQ)\xrightarrow{\Sq^2} h^{3,6}(\QQ) \xrightarrow{\Sq^4} h^{7,8}(\QQ)\to H^{8,8}(\QQ)\]
is zero.
Note that the sum of these two nonzero elements is $\partial^2_\infty\Sq^2\Sq^4\pr^\infty_2$ by an Adem relation \cite{Voevodsky_Steenrod, HKO_Steenrod}, which implies also that the last short exact sequence is split.
\end{proof}

\begin{lemma}\label{lem:Sq6_tau4}
    We have $\Sq^6(\tau^4c)=\rho^6\tau c$ for every $c\in \KMil_\star(F)$.
\end{lemma}
\begin{proof}
This follows from the motivic Cartan formula \cite[Proposition 9.7]{Voevodsky_Steenrod}, applied first to the pair $(\tau^2,\tau^2)$ to deduce $\Sq^6(\tau^4)=\rho^6\tau$ from $\Sq^3(\tau^2)=\rho^3$, and then to the pair $(\tau^4,c)$, using that all Steenrod squares except $\Sq^0$ vanish on $c\in \KMil_\star(F)$.
\end{proof}

\begin{lemma}\label{lem:d3pi4s0MSL}
    The second component of the differential $\dd_3\colon E^3_{-n+4,0,-n}(\MSL_\Lambda)\to E^3_{-n+3,3,-n}(\MSL_\Lambda)$ is nonzero in general. It is given by the operation
    \[ \partial^2_\infty\Sq^6\pr^\infty_2\colon H^{n-4,n}\to H^{n+3,n+3}. \]
\end{lemma}
\begin{proof}
Again $\Lambda$ will be ignored in the notation. By a slight abuse of terminology, we refer to the second component simply as the $\dd_3$-differential and denote it by $\dd_3^{\MSL}$ (recall that the first component is trivial by Corollary \ref{cor:d3pi4s0msl}).
Consider the spectrum $\MGL/(a_1,a_2)$, where $a_1$ and $a_2$ are the first two polynomial generators of the Lazard ring. The slices of this spectrum are readily computed. Lemma \ref{lemma:xi_decomposition_ai} implies that the map $\slice_3\MSL\to\slice_3\MGL/(a_1,a_2)$ within the integral summands is the multiplication by $-3$. Combining this with Corollary \ref{cor:d3pi4s0msl}, we obtain the commutative diagram
\begin{center}
    \begin{tikzcd}
        H^{n-4,n} \ar[r, "\simeq"] \ar[d, "\dd_3^\MSL"'] & H^{n-4,n} \ar[d, "\dd_3^{\MGL/(a_1,a_2)}"] \\ H^{n+3,n+3} \ar[r, "-3"] & H^{n+3,n+3},
    \end{tikzcd}
\end{center}
which implies that the desired $\dd_3$-differential is equal to $-3\dd_3^{\MGL/(a_1,a_2)}$. We claim that $\dd_3^{\MGL/(a_1,a_2)}=\partial^2_\infty\Sq^6\pr^\infty_2$.
First, let us prove that the claim implies the result. The exact sequence
\[ 0\to {}_3H^{\star+3,\star+3}\to H^{\star+3,\star+3}\xrightarrow{\cdot 3} H^{\star+3,\star+3} \]
of graded $\KMil_\star(F)$-modules induces the exact sequence
\[ 0\to \Hom_{\KMil_\star(F)}(H^{\star-4,\star},{}_3H^{\star+3,\star+3})\to\Hom_{\KMil_\star(F)}(H^{\star-4,\star},H^{\star+3,\star+3})\to \Hom_{\KMil_\star(F)}(H^{\star-4,\star},H^{\star+3,\star+3}). \]
It follows from the claim that the last displayed homomorphism sends $\dd_3^\MSL$ to $\partial^2_\infty\Sq^6\pr^\infty_2$.
Since $3\partial^2_\infty\Sq^6\pr^\infty_2=\partial^2_\infty\Sq^6\pr^\infty_2$, the differential $\dd_3^\MSL$ has the form $\partial^2_\infty\Sq^6\pr^\infty_2+\phi$ for some map of $\KMil_\star(F)$-modules $\phi$ that factors through ${}_3H^{\star+3,\star+3}\cong{}_3\KMil_{\star+3}(F)$. Assume that $\phi$ is nonzero. Then the composite of the desired differential with $\pr^\infty_3$ is non-trivial (the $\KMil_\star(F)$-module ${}_3\KMil_{\star+3}(F)$ injects into $\KMil_{\star+3}(F)/3$). This non-trivial map is the composite via the bottom left in the commutative diagram 
\begin{center}
    \begin{tikzcd}
        H^{\star-4,\star} \ar[rr, "\pr^\infty_3"] \ar[d, "\dd_3^{\MSL}"'] & & h^{\star-4,\star}_3 \ar[d, "\dd_3^{\MSL/3}"] \\ H^{\star+3,\star+3} \ar[rr, "\pr^\infty_3"] & & h^{\star+3,\star+3}_3, 
    \end{tikzcd}
\end{center}
which is induced by the morphism $\MSL\to \MSL/3$. However, the differential $\dd_3^{\MSL/3}$ is trivial. Indeed, if $F$ contains a primitive 3rd root of unity, this differential is determined by the image of any generator $g\in h^{0,4}_3$, which exists over the subfield $F_0(\zeta_3)\subset F$ of $3$-cohomological dimension at most one. If $F$ does not contain a primitive 3rd root of unity, the degree two field extension $F\subset F(\zeta_3)$ induces injections on the groups occurring in the differential.
This yields a contradiction, and thus we have $\phi=0$.
 
Now we prove the claim. The $\dd_3$-differential for $\MGL/(a_1,a_2)$ comes from a map \[\mathbf{d}_3^{\MGL/(a_1,a_2)}\colon\EM\Lambda\to \Sigma^{7,3}\EM\Lambda\] since $\slice_1\MGL/(a_1,a_2)$ and $\slice_2\MGL/(a_1,a_2)$ vanish. 
For the convenience of the reader, we distinguish the map $\mathbf{d}_3^{\MGL/(a_1,a_2)}$ (the differential as a map of motivic spectra) from $\dd_3^{\MGL/(a_1,a_2)}$ (the induced map in Milnor--Witt degree $4$, i.e., $\dd_3^{\MGL/(a_1,a_2)}=\pi_{4-(\star)}(\mathbf{d}_3^{\MGL/(a_1,a_2)})$). 
By Lemma \ref{lem:MZtoMZ(3)[7]}, there exist $\alpha,\beta\in\mathbb{F}_2$ such that
\[ \mathbf{d}_3^{\MGL/(a_1,a_2)}=\alpha\partial^2_\infty\Sq^6\pr^\infty_2+\beta\partial^2_\infty\Sq^4\Sq^2\pr^\infty_2. \]
Whatever $\beta$ is, the equality $\dd_3^{\MGL/(a_1,a_2)}=\alpha\partial^2_\infty\Sq^6\pr_2^\infty$ holds because $\partial^2_\infty\Sq^4\Sq^2\pr^\infty_2$ acts trivially on $H^{\star-4,\star}$.
In fact, already $\Sq^2\colon h^{\star-4,\star}\to h^{\star-2,\star+1}$ is trivial. 
Hence, there are only two options for $\dd_3^{\MGL/(a_1,a_2)}$: it is either trivial or given by $\partial^2_\infty\Sq^6\pr^\infty_2$. 
If $F$ has odd characteristic, these options coincide by Lemma~\ref{lem:Sq6_tau4}, since in this case $\rho^2=0$. 
To prove the claim in characteristic zero, it remains to show that the differential is nonzero over $\QQ$. 
For that consider the map $\MGL/(a_1,a_2)\to \MGL/(2,a_1,a_2)$. 
It induces the following commutative diagram
\begin{center}
    \begin{tikzcd}
        H^{\star-4,\star} \ar[rr, "\pr^\infty_2"] \ar[d, "\dd_3^{\MGL/(a_1,a_2)}"'] & & h^{\star-4,\star} \ar[d, "\dd_3^{\MGL/(2,a_1,a_2)}"] \\ H^{\star+3,\star+3} \ar[rr, "\pr^\infty_2"] & & h^{\star+3,\star+3}.
    \end{tikzcd}
\end{center}
The right vertical map sends the unique non-trivial element $g\in h^{0,4}$ to $\rho^7$ over $\RR$ (and, hence, over $\QQ$) by \cite[Theorem~5.4]{MGL_over_R}. 
If the left differential was zero, then the composition via the top right would be zero, but this cannot happen.
For example, the element $\rho g\in h^{1,5}(\QQ)$ belongs to the image of the top homomorphism (see the end of the proof of Lemma \ref{lem:d2pi4s0sphere}) and maps to $\rho^8\neq 0\in h^{8,8}(\QQ)$ under the right differential. 
It follows that the differential $\dd_3^{\MGL/(a_1,a_2)}$ is given by $\partial^2_\infty\Sq^6\pr^\infty_2$ over $\QQ$. 
This completes the proof.
\end{proof}

\begin{lemma}\label{lem:stabilization_m=3}
    The equality $E^4_{-n+3,q,-n}(\MSL_\Lambda)=E^\infty_{-n+3,q,-n}(\MSL_\Lambda)$ holds for all integers $n$ and $q$.
    The same holds for $\kq_\Lambda$ instead of $\MSL_\Lambda$, and the morphism $\MSL_\Lambda\to \kq_\Lambda$ induces isomorphisms
    \[ E^\infty_{-n+3,q,-n}(\MSL_\Lambda)\xrightarrow{\simeq}E^\infty_{-n+3,q,-n}(\kq_\Lambda)\ \text{for}\ q\neq 3, \]
    and a split epimorphism $E^\infty_{-n+3,3,-n}(\MSL_\Lambda)\twoheadrightarrow E^\infty_{-n+3,3,-n}(\kq_\Lambda)$ with kernel 
    \[ H^{n+3,n+3}/(\partial^2_\infty\Sq^6\pr^\infty_2H^{n-4,n}). \]
\end{lemma}
\begin{proof}
    The stabilization follows from the vanishing stated in Lemma \ref{lemma:E^2_for_m<3} and Proposition \ref{prop:E^2_for_m=3}. For the comparison with $\kq_\Lambda$, combine Proposition \ref{prop:E^2_for_m=3} with the computations of higher slice differentials obtained in Corollary \ref{cor:kq-msl-d2pi4s0}, Lemma \ref{lem:d2pi4s1sphere}, Lemma \ref{lem:d3pi4s0kq}, Corollary \ref{cor:d3pi4s0msl}, and Lemma~\ref{lem:d3pi4s0MSL}.
\end{proof}

\begin{theorem}[Dugger--Isaksen]\label{thm:pi3-(5)reals}
  Let $\sigma\in \pi_{3+(4)}\sph$ be the element given by the third motivic Hopf map.
  There is an isomorphism
  \[ \pi_{3-(5)}\sph^{\wedge}_{\eta}(\RR)\iso \Z/8\Z\{\rho^9\sigma\}\oplus H^{2,5}(\RR)[\nicefrac{1}{2}],\]
  where the second factor is a uniquely divisible group containing no $2$-torsion.
\end{theorem}

\begin{proof}
  The main result of \cite{dugger-isaksen.real} computes $\pi_m\sph^{\wedge}_{\eta,2}(\RR)$ for $m\in \{0,1,2,3\}$ using the Adams spectral sequence.
  In particular, $\pi_{3-(5)}\sph^{\wedge}_{\eta,2}(\RR)\iso \Z/8\Z \{\rho^9\sigma\}$.
  The arithmetic square for the prime $2$ gives an exact sequence
  \[ 0=\pi_4\sph^{\wedge}_{\eta,2}[\nicefrac{1}{2}]\to \pi_{3}\sph^{\wedge}_{\eta}\to
    \pi_{3}\sph^{\wedge}_{\eta,2}\oplus \pi_3\sph^{\wedge}_{\eta}[\nicefrac{1}{2}] \to
    \pi_3\sph^{\wedge}_{\eta,2}[\nicefrac{1}{2}]=0 \]
  of Milnor--Witt modules in which the outer terms are zero for $\RR$.
  The reference for this vanishing in degree $3$ is \cite{dugger-isaksen.real}, the reference in degree $4$ is the less explicit \cite{belmont-isaksen}.
  It remains to identify $\pi_{3-(5)}\sph^{\wedge}_{\eta}[\nicefrac{1}{2}]$ for $\RR$.
  The second page of the slice spectral sequence for $\pi_{3-(5)}$ over $\RR$ contains only finite $2$-torsion groups in positive slice degrees, and the group $H^{2,5}(\RR)$ in slice degree $0$.
  After passage to the fourth page, the resulting column is finite by Lemma~\ref{lem:unit-d3pi4s5}.
  Hence $\pi_{3-(5)}\sph^{\wedge}_{\eta}(\RR)$ is an extension of $H^{2,5}$ and a finite $2$-primary torsion group.
  Inverting $2$ thus produces an isomorphism $\pi_{3-(5)}\sph^{\wedge}_{\eta}(\RR)[\nicefrac{1}{2}] \iso H^{2,5}(\RR)[\nicefrac{1}{2}]$ induced by mapping to the zero slice.
  Since motivic cohomology with odd prime coefficients for $\RR$ is zero in positive degree and any weight, this group is uniquely divisible.
\end{proof}

\section{Computations of Milnor--Witt stems for special linear algebraic cobordism}\label{section:6}
In this section, we determine the first four non-trivial Milnor--Witt stems of $\MSL_\Lambda$ via the map $\MSL_\Lambda\to\kq_\Lambda$. 
The result for $\pi_0$ is well-known, and the result for $\pi_1$ was proven in work of the first author \cite{NanMSL}.
In order to employ the computations from Section~\ref{section:5}, convergence properties of the slice spectral sequence for $\MSL$ have to be discussed. 
In order to do so, recall a few definitions from \cite[\S 3.1]{RSO19}, stated for an arbitrary motivic spectrum $\EE\in\SH(F)$, but eventually applied to the cases $\EE\in\{\MSL,\kq\}$.
\begin{definition}
    The \textit{slice completion of a motivic spectrum $\EE$} is the motivic spectrum $\mathsf{sc}(\EE)\in\SH(F)$, defined as the (homotopy) limit
        \[ \lim_{q\to\infty}\mathrm{cofib}(\effcover_{q}(\EE)\to\EE). \]
    By construction, the slice spectral sequence converges conditionally to the homotopy groups of $\mathsf{sc}(\EE)$. 
\end{definition}
\begin{definition}
    The \textit{$\eta$-completion of a motivic spectrum $\EE$} is the motivic spectrum $\EE^\wedge_\eta\in\SH(F)$, defined as the (homotopy) limit of the inverse system
        \[ \dots\to \EE/\eta^{3}\to \EE/\eta^2\to\EE/\eta. \]
\end{definition}
\begin{remark}\label{remark:eta-arithm_square}
    The spectrum $\EE^\wedge_\eta$ constitutes one corner of the (homotopy) pullback square 
    \begin{center}
    \begin{tikzcd}
        \EE \ar[r] \ar[d] & \EE[\eta^{-1}] \ar[d] \\ \EE^\wedge_\eta \ar[r] & \EE^{\wedge}_\eta[\eta^{-1}],
    \end{tikzcd}
    \end{center}
    which is called the \textit{$\eta$-arithmetic square of $\EE$} (see \cite[Lemma 3.9]{RSO19}). 
    It induces the long exact sequence
    \begin{equation}\label{equation:les_eta_arithm}
        \dots\to \pi_{n+1}\EE^\wedge_\eta[\eta^{-1}]\to\pi_n\EE\to\pi_n\EE^\wedge_\eta\oplus\pi_n\EE[\eta^{-1}]\to \pi_n\EE^\wedge_\eta[\eta^{-1}]\to \cdots. 
    \end{equation}
\end{remark}
\begin{lemma}\label{lemma:convergence}
    There is a canonical equivalence $\mathsf{sc}(\MSL)\cong\MSL^\wedge_\eta$. 
    In other words, the slice spectral sequence for $\MSL$ converges conditionally to the homotopy groups of $\MSL^\wedge_\eta$.
\end{lemma}
\begin{proof}
    By \cite[Lemma 3.13]{RSO19} it is enough to prove that $\MSL/\eta\cong\MWL$ is $\eta$-complete. 
    This follows from the cofiber sequence (2) of Theorem \ref{MWL_cofib} since $\MGL$ is $\eta$-complete.
\end{proof}
\begin{remark}
    A similar convergence property holds for $\kq$. 
    In other words, the slice spectral sequence for $\kq$ converges conditionally to the homotopy groups of $\kq^\wedge_\eta$, see \cite[Theorem 4.1]{very_eff_KQ}.
\end{remark}
\begin{lemma}\label{lemma:vanishing_etacompl_etainv}
    The vanishing $\pi_n(\MSL_\Lambda)^\wedge_\eta[\eta^{-1}]=0$ holds for $1\leq n\leq 3$. The same is true for $\kq_\Lambda$.
\end{lemma}
\begin{proof}
    This follows directly from Lemma~\ref{lemma:E^2_for_m<3},  Proposition~\ref{prop:E^2_for_m=3}, and Lemma~\ref{lemma:convergence}, since the groups on the $E^2$-page that contribute to $\pi_n$ are trivial in slice degrees $\geq 4$.
\end{proof}
\begin{lemma}\label{lemma:pi_1,2_etacompl}
    The natural map $\MSL_\Lambda\to(\MSL_\Lambda)^\wedge_\eta$ induces isomorphisms
    \[ \pi_{n}\MSL_\Lambda\xrightarrow{\simeq}\pi_n(\MSL_\Lambda)^\wedge_\eta\ \text{for}\ n=1,2. \]
    The same holds for $\kq_\Lambda$.
\end{lemma}
\begin{proof}
    Combining the vanishing stated in Lemma \ref{lemma:vanishing_etacompl_etainv} with the long exact sequence \eqref{equation:les_eta_arithm} gives an isomorphism
    \[ \pi_n\MSL_\Lambda\cong \pi_n(\MSL_\Lambda)^\wedge_\eta\oplus\pi_n\MSL_\Lambda[\eta^{-1}]\]
    for $n\in \{1,2\}$, and similarly for $\kq_\Lambda$.
    The $n$-th homotopy group of $\MSL_\Lambda[\eta^{-1}]$ vanishes by the Bachmann--Hopkins result \cite[Corollary 1.3(3)]{BH_eta}. 
    The corresponding vanishing for $\kq_\Lambda[\eta^{-1}]=\mathbf{kw}_\Lambda$ is well-known, see e.g., \cite[6.3.2]{BH_eta}.
\end{proof}
\begin{theorem}\label{thm:pi_2MSL}
    Let $F$ be a field of characteristic $\neq 2$ and $\Lambda=\Z[\nicefrac{1}{e}]$, where $e$ is the exponential characteristic of $F$. 
    The morphism $\MSL\to\kq$ induces isomorphisms
    \[ \pi_n\MSL_\Lambda\xrightarrow{\simeq}\pi_n\kq_\Lambda\ \text{for}\ n\leq 2. \]
\end{theorem}
\begin{proof}
    The statement for $\pi_0$ follows from comparison with the motivic sphere spectrum $\sph$, see \cite[Example 16.35]{BH-Norms} and \cite[Theorem 4.2]{very_eff_KQ}. 
    For $n\in \{1,2\}$, this is a consequence of Lemmas \ref{lemma:E^2_for_m<3}, \ref{lemma:stab_m<3}, \ref{lemma:convergence}, and \ref{lemma:pi_1,2_etacompl}.
\end{proof}
Now we turn to the computation of $\pi_3\MSL_\Lambda$. 
Our first goal is to prove that the boundary homomorphism 
\[\pi_4(\MSL_\Lambda)^{\wedge}_{\eta}[\eta^{-1}]\to \pi_3\MSL_{\Lambda}\]
in the long exact sequence \eqref{equation:les_eta_arithm} for $\MSL_\Lambda$ is zero. 
The argument is similar to the one given in \cite[\S 5]{RSO24} for the motivic sphere spectrum $\sph_\Lambda$ and $n=2$. 
\begin{lemma}\label{lemma:surj1}
    Multiplication with the Hopf element $\eta$ induces a surjection 
    \[\pi_{4+(n)}(\MSL_\Lambda)^\wedge_\eta\to \pi_{4+(n+1)}(\MSL_\Lambda)^\wedge_\eta\]
    for $n\geq 4$, which moreover is an isomorphism for $n>4$.
\end{lemma}
\begin{proof}
    Consider the entries on the second page of the slice spectral sequence that contribute to the group $\pi_{4+(n)}\MSL_\Lambda$ for $n\geq 4$. 
    They are trivial in slice degrees $0\leq q\leq n-1$. 
    Moreover, since $E^2_{n+3,q,n}(\MSL_\Lambda)$ vanishes for $q>3$ by Proposition \ref{prop:E^2_for_m=3}, the remaining groups consist of permanent cycles.
    
    We claim that multiplication with the motivic Hopf element induces a surjection between the 4th columns of the $E^\infty$-page
    \[ E^\infty_{n+4,\star,n}(\MSL_\Lambda)\xrightarrow{\eta} E^\infty_{n+5,\star+1,n+1}(\MSL_\Lambda) \]
    for $n\geq 4$, which is even an isomorphism for $n>4$.
    To prove the statement about the $E^\infty$-page, it suffices to prove it for the $E^r$-page for any $r\geq 2$. 
    Consider the following commutative diagram:
    \begin{center}\begin{tikzcd} 
    E^r_{n+5,\star-r,n} \ar[r,"\dd_r"] \ar[d,"\eta"'] & E^r_{n+4,\star,n} \ar[r] \ar[d,"\eta"] & E^{r+1}_{n+4,\star,n} \ar[d,"\eta"] \ar[r] & 0 \\ E^r_{n+6,\star-r+1,n+1} \ar[r,"\dd_r"] & E^r_{n+5,\star+1,n+1} \ar[r] & E^{r+1}_{n+5,\star+1,n+1} \ar[r] & 0
    \end{tikzcd}\end{center}
    A straightforward computation with Lemma \ref{lemma_eta_action} shows that for $r=2$ the vertical homomorphism in the middle is surjective for $n\geq 4$, while the vertical homomorphism on the left hand side is surjective for $n>4$. 
    Moreover, for $n>4$ the vertical homomorphism in the middle is an isomorphism by Lemma \ref{lemma:E2_4>4}. 
    By induction on $r$ and the Five Lemma, the vertical homomorphism in the middle is surjective for $n\geq 4$ and an isomorphism for $n>4$, while the vertical homomorphism on the left hand side is surjective for $n>4$. 
    The claim for the $E^\infty$-page follows.
    Applying Lemma \ref{lemma:convergence} supplies the result.
\end{proof}

\begin{lemma}\label{lemma:surj2}
    The homomorphism $\pi_{4+(n)}(\MSL_\Lambda)^\wedge_\eta\to \pi_{4+(n)}(\MSL_\Lambda)^\wedge_\eta[\eta^{-1}]$ is surjective for $n\geq 4$ and an isomorphism for $n>4$.
\end{lemma}
\begin{proof}
    This follows directly from Lemma \ref{lemma:surj1}.
\end{proof}
\begin{lemma}\label{lem:coker_W/I^}
    The cokernel of the morphism $\pi_{4+(n)}(\MSL_\Lambda)^\wedge_\eta\to \pi_{4+(n)}(\MSL_\Lambda)^\wedge_\eta[\eta^{-1}]$ is a quotient of $\bigl(\Witt(F)/\mathbf{I}(F)^{-(n-4)}\bigr)\otimes \Lambda$ for all $n\in \Z$.\footnote{Here $\mathbf{I}(F)^k$ denotes the $k$-th power of the fundamental ideal $\mathbf{I}(F)\subset \Witt(F)$, which by definition coincides with $\Witt(F)$ for nonpositive $k$.}
\end{lemma}
\begin{proof}
    For simpler notation, the coefficient ring $\Lambda$ will be omitted in the proof.
    For $n\geq 5$, the column of the $E^2$-page of the $n$-th slice spectral sequence computing $\pi_{4+(n)}\MSL^\wedge_\eta$ does not depend on $n$. 
    It consists of the groups $h^{0,0},h^{1,1},h^{2,2},\dots$ (see Lemma \ref{lemma:E2_4>4}), which form the associated graded of the $\mathbf{I}(F)$-adic filtration on the Witt ring $\Witt(F)$ by Milnor's conjecture on quadratic forms \cite{OVV_Milnor_qf} (see also \cite{RO16}). 
    It follows that $\pi_{4+(n)}\MSL^\wedge_\eta$ is a quotient of the $\mathbf{I}(F)$-adic completion $\Witt(F)^\wedge_{\mathbf{I}(F)}$, and the same holds for $\pi_{4+(n)}\MSL^\wedge_\eta[\eta^{-1}]$. 
    Thus, the desired cokernel is also a quotient of $\Witt(F)^\wedge_{\mathbf{I}(F)}$, which is trivial for $n\geq 4$ by Lemma \ref{lemma:surj2}. 
    For $n<4$, multiplication with $\eta^{5-n}$ hits the groups $h^{-n+4,-n+4},h^{-n+5,-n+5},\dots$, but not $h^{0,0},h^{1,1},\dots,h^{-n+3,-n+3}$. Therefore, the cokernel of the morphism
    \[ \pi_{4+(n)}\MSL^\wedge_\eta\xrightarrow{\eta^{5-n}}\pi_{4+(5)}\MSL^\wedge_\eta\cong \pi_{4+(5)}\MSL^\wedge_\eta[\eta^{-1}]\cong \pi_{4+(n)}\MSL^\wedge_\eta[\eta^{-1}]\]
    is a quotient of $\Witt(F)/\mathbf{I}(F)^{-(n-4)}$, as required. 
\end{proof}

\begin{proposition}\label{prop:eta_arithmetic_bound_zero}
    The canonical map $\pi_{4}(\MSL_\Lambda)^\wedge_\eta\oplus \pi_{4}{\MSL}_\Lambda[\eta^{-1}] \to \pi_{4}(\MSL_\Lambda)^\wedge_\eta[\eta^{-1}]$ is surjective.
\end{proposition}
\begin{proof}
    The proof ignores $\Lambda$ in the notation. 
    Lemma \ref{lem:coker_W/I^} implies that the cokernel of 
    \[ \pi_{4+(n)}{\MSL}^\wedge_\eta\to \pi_{4+(n)}{\MSL}^\wedge_\eta[\eta^{-1}] \]
    is generated (as a $\KMW_\star(F)$-module) by a single element. 
    To construct a generator, use \cite[Theorem D]{ZolMSL}, which gives a cartesian square
    \begin{center}\begin{tikzcd} 
    \pi_{4+(4)}\MSL \ar[r] \ar[d] & \Z\{\textstyle \frac{1}{4} y_2^2\}\oplus \Z\{y_4\} \ar[d] \\ \mathbf{W}(F) \ar[r,"\overline{\mathrm{rk}}"] & \Z/2\Z
    \end{tikzcd}\end{center} 
    where $\frac{1}{4}y_2^2=x_1^{*4}$ and $y_4=2x_4-x_1*x_3$ are generators of $\MSU_8$, the vertical map on the left hand side is given by 
    \[\pi_{4+(4)}\MSL \xrightarrow{\eta}\pi_{4+(5)}\MSL \cong\pi_{4+(5)}\MSL[\textstyle \eta^{-1}]\cong\Witt(F)\] 
    and the vertical homomorphism on the right hand side is the following composition:
    \[\MSU_8 \xrightarrow{\etatop}\MSU_9\cong \Z/2\Z.\] By \cite[\S 7]{CLP} the latter sends $y_4$ to zero and $\frac{1}{4}y_2^2$ to the non-trivial element of $\Z/2\Z$. 
    There results an isomorphism 
    \[\pi_{4+(4)}\MSL\cong \GWitt(F)\{\textstyle \frac{1}{2\hyper}y_2^2\}\oplus \Z\{y_4\},\]
    where $\frac{1}{2\hyper}y_2^2=(\frac{1}{4}y_2^2,1)$ and $y_4=(y_4,0)$ with respect to the cartesian square above. 
    The notation $\frac{1}{2\hyper}y_2^2$ indicates that this element satisfies the relation $2\hyper\cdot (\frac{1}{2\hyper}y_2^2)=y_2^2,$ where $y_2$ is the generator of $\pi_{2+(2)}\MSL\cong \Z$ that corresponds under the map $\MSL\to\mathbf{KQ}$ to the standard skew-symmetric form in $\pi_{2+(2)}\mathbf{KQ}\cong \mathbf{K}^{\Sp}_0(F)$.
    From this perspective the generator $\frac{1}{2\mathrm{h}}y_2^2$ maps to the Karoubi--Bott periodicity element in $\pi_{4+(4)}\mathbf{KQ}$. 
    We claim that $\frac{1}{2\hyper}y_2^2$ is the desired generator. 
    Indeed, for weight reasons it lifts (uniquely) to $\pi_{4+(4)}\effcover_4\MSL$, and this lift maps to a generator of $\pi_{4+(4)}\slice_4(\MSL)\cong\Z^2$ whose $\eta$-multiple is non‑trivial. 
    It follows from Lemma \ref{lemma:surj2} and Lemma \ref{lem:coker_W/I^} that it must generate the desired cokernel. 
    It remains to note that the generator of $\pi_{4+(4)}\MSL[\eta^{-1}]\cong \Witt(F)$ maps to the image of $\frac{1}{2\hyper}y_2^2$ by construction.
\end{proof}
\begin{corollary}\label{cor:pi_3=with_eta_comp}
    The canonical map $\pi_3\MSL_\Lambda\to \pi_3(\MSL_\Lambda)^\wedge_\eta$ is an isomorphism.
\end{corollary}
\begin{proof}
    The long exact sequence \eqref{equation:les_eta_arithm}, Lemma \ref{lemma:vanishing_etacompl_etainv} and Proposition \ref{prop:eta_arithmetic_bound_zero} imply that the map 
    \[ \pi_3\MSL_\Lambda\to \pi_3(\MSL_\Lambda)^\wedge_\eta\oplus \pi_{3}{\MSL}_\Lambda[\eta^{-1}] \]
    is an isomorphism. 
    The result follows from the vanishing $\pi_{3}{\MSL}_\Lambda[\eta^{-1}]=0$ \cite[Corollary 1.3]{BH_eta}.
\end{proof}
\begin{remark}\label{rem:pi_3=with_eta_comp_kq}
    The same argument shows that the map $\pi_3\kq_\Lambda\to\pi_3(\kq_\Lambda)^\wedge_\eta$ is an isomorphism.
\end{remark}
\begin{theorem}\label{thm:pi_3MSL}
    Let $F$ be a field of characteristic $\neq 2$ and $\Lambda=\Z[\nicefrac{1}{e}]$, where $e$ is the exponential characteristic of $F$. 
    The morphism $\MSL\to \kq$ induces a short exact sequence
    \[ 0\to\bigl(\KMil_{3-\star}(F)/(\partial^2_\infty\Sq^6\pr^\infty_2H^{-\star-4,-\star})\bigr)\otimes\Lambda\to \pi_{3+(\star)}\MSL_\Lambda\to \pi_{3+(\star)}\kq_\Lambda\to 0. \]
\end{theorem}
\begin{proof}
    This follows directly from Lemma \ref{lem:stabilization_m=3}, Lemma \ref{lemma:convergence}, Corollary \ref{cor:pi_3=with_eta_comp}, and Remark \ref{rem:pi_3=with_eta_comp_kq}.
\end{proof}
\begin{remark}
    By Morel's theorem \cite[Theorem 2.11]{Morel_hm}, the stated results are valid on the level of homotopy sheaves as well. 
    In particular, Theorem \ref{thm:pi_2MSL} and Theorem \ref{thm:pi_3MSL} imply that the fiber of the map $\MSL_\Lambda\to\kq_\Lambda$ is exactly $3$-connective with respect to the homotopy $t$-structure.
\end{remark}

\section{A rational decomposition of special linear algebraic cobordism}
Since $2$ is invertible in $\MSL\otimes \QQ$, it splits naturally as
$(\MSL\otimes \QQ)^+\oplus (\MSL\otimes \QQ)^-$.
The first summand can be identified with $\eta$-completion, while the second corresponds to $\eta$-inversion.
In this section, we use slices to prove that $(\MSL\otimes\QQ)^+$ is a polynomial algebra over the rational motivic cohomology spectrum $\EM\QQ$. 
The summand $(\MSL\otimes\QQ)^-$ may be treated using the work of Bachmann \cite{Bachmann_real_etale}. 
First, we construct the corresponding polynomials generators using complex and real Betti realization. 
The values of these realization functors on $\MSL$ are given by the special unitary and special orthogonal bordism spectra, respectively; see for example \cite[Theorem A]{real_realization}.

\begin{lemma}\label{lem:geom_diag_overZ}
    The complex and real Betti realizations of $\MSL\otimes\QQ\in\SH(\Z)$ induce an isomorphism
    \[ \pi_{\star+(\star)}\MSL\otimes\QQ\xrightarrow{\simeq} (\pi_{2\star}\MSU\otimes\QQ)\times (\pi_{\star}\mathbf{MSO}\otimes\QQ) \]
    of graded rings.
\end{lemma}
\begin{proof}
    By \cite[Theorem D]{ZolMSL}, the complex and real Betti realizations induce the desired isomorphism for $\MSL\otimes\QQ\in \SH(\QQ)$. Hence, it suffices to show that base change along the ring inclusion $\Z\to \QQ$ induces an isomorphism on $\pi_{\star+(\star)}\MSL\otimes\QQ$.

    First, we claim that $\pi_{p+(q)}\MWL\otimes\QQ=0$ for $p<q$ both over $\Z$ and $\QQ$. Over $\QQ$ this is proven in \cite[Proposition 4.8]{ZolMSL} even integrally. The proof works over $\Z$ almost in the same way, with the difference in the proof of the vanishing of $\pi_{p+(q)}\MGL\otimes\QQ$ for $p<q$. To show this, combine the convergence of the slice filtration for $\MGL\otimes\QQ\in\SH(\Z)$ (see e.g., the beginning of the proof of \cite[Theorem B.1]{Chowt}) together with the vanishing of the corresponding rational motivic cohomology groups $H^{*,*}(\Z,\QQ)$ stated in \cite[Lemma 5.3]{KRO_hermitian_k_theory_char2}.

    Applying the claim to the exact sequence of homotopy groups associated with the first cofiber sequence from Theorem \ref{MWL_cofib}, we obtain the $\eta$-stabilization
    \[ \pi_{n+(n+1)}\MSL\otimes\QQ\xrightarrow{\simeq} \pi_{n+(n+1)}\MSL[\eta^{-1}]\otimes\QQ=\pi_{n+(n+1)}(\MSL\otimes\QQ)^- \]
    for $n\in\Z$ (cf. \cite[Theorem 5.2]{ZolMSL}). 
    Now the same cofiber sequence induces the exact sequence 
    \[ \pi_{n+1+(n+1)}\MSL\otimes\QQ\to \pi_{n+1+(n+1)}\MWL\otimes\QQ\to\pi_{n+(n)}\MSL\otimes\QQ\to\pi_{n+(n+1)}(\MSL\otimes\QQ)^-\to 0, \]
    both over $\Z$ and $\QQ$. 
    We claim that base change along $\Z\to \QQ$ induces an isomorphism on these sequences. 
    The fourth group maps isomorphically by the result of Bachmann \cite{Bachmann_real_etale} since $\mathrm{Sper}(\Z)=\mathrm{Sper}(\QQ)$ (note that inverting $\rho$ coincides with inverting $\eta$ whenever $2$ is already invertible).
    The second group maps isomorphically by an argument similar to \cite[Proposition 4.10]{ZolMSL}, which uses \cite[Theorem 6.5, Proposition 7.7]{SpiMGL2}, $\Pic(\Z)=0$, and $H^{3,2}(\Z,\QQ)=0$. 
    Applying the Five lemma twice provides the result.
\end{proof}

Combining Lemma \ref{lem:geom_diag_overZ} with standard topological computations, we obtain that the geometric diagonal of $\MSL\otimes\QQ\in \SH(\Z)$ is given by
\[ \pi_{\star+(\star)}\MSL\otimes\QQ\cong \QQ[y_2,y_3,\dots]\times \QQ[u_4,u_8,\dots], \]
where the generator $y_i$ has bidegree $i+(i)$ and the generator $u_{4i}$ has bidegree $4i+(4i)$. 
Here the first polynomial ring is the geometric diagonal of $(\MSL\otimes\QQ)^+$, and the second one is the geometric diagonal of $(\MSL\otimes\QQ)^-$. 
We begin with the decomposition of the plus part.

Let $I=(i_1,i_2,\dots,i_k)$ be a partition of $n$ such that $i_j\neq 1$, and consider the monomial 
\[ y_I=y_{i_1}y_{i_2}\dots y_{i_k}\colon \Sigma^{n+(n)}\sph\to(\MSL\otimes\QQ)^+. \]
This map factors (uniquely) through $\Sigma^{n+(n)}(\sph\otimes\QQ)^+$. Using $(\sph\otimes\QQ)^+\cong \EM\QQ$ (see \cite[Theorem~11]{Cisinski-Deglise}), we obtain the following map in $\SH(\Z)\otimes\QQ$:
\[ y_I\colon\Sigma^{n+(n)}\EM\QQ\to (\MSL\otimes\QQ)^+. \]
\begin{notation} The direct sum of these morphisms along all partitions with no part equal to $1$ produces a map
\[ \varphi^+\colon\bigoplus_{n=0}^\infty\Sigma^{n+(n)}\EM\QQ^{p(n)-p(n-1)}\to (\MSL\otimes\QQ)^+. \]
Endow the source of $\varphi^+$ with an obvious homotopy commutative ring structure of a polynomial algebra over $\EM\QQ$ with generators $y_i$ for $i\geq 2$. 
Then $\varphi^+$ becomes a morphism of homotopy commutative ring spectra. 
Given a qcqs scheme $S\to\Spec(\Z)$, pulling back $\varphi^+$ yields a morphism of homotopy commutative ring spectra in $\SH(S)\otimes\QQ$:
\[ \varphi^+_S\colon\EM\QQ[y_2,y_3,\dots]\to (\MSL\otimes\QQ)^+. \]
\end{notation}
\begin{proposition}\label{prop:Q_plus_part}
    The map $\varphi^+_S\colon\EM\QQ[y_2,y_3,\dots]\to (\MSL\otimes\QQ)^+$ is an isomorphism.
\end{proposition}
\begin{proof}
    By \cite[Proposition B.3]{BH-Norms}, it suffices to consider the case $S=\Spec(F)$ for a field $F$, since $\varphi^+$ is stable under base change (in order to remove the assumption that $S$ is locally of finite Krull dimension, one may first reduce by base change to $\Spec(\Z)$ and then to its residue fields). 
    For brevity, denote by $A$ our homotopy polynomial algebra $\EM\QQ[y_2,y_3,\dots]$ and by $A_q$ the degree $q$ part of $A$, i.e., the direct sum
    $$ \bigoplus_{|I|=q}\Sigma^{q+(q)}\EM\QQ\{y_I\}. $$
    Note that the slices of $(\MSL\otimes\QQ)^+$ are canonically equivalent to those of $\MSL\otimes\QQ$ by \cite[Lemma 3.6]{RSO19}.
    
    By construction, the map $\varphi_F^+\colon A\to (\MSL\otimes\QQ)^+$ induces an equivalence  
    \[ A_q=\slice_q(A)\cong \slice_q(\MSL\otimes\QQ)^+. \]
    Denote by $\varphi_{F,q}^+$ the composition of the split monomorphism $A_q\to A$ with $\varphi_F^+$. 
    This map factors through $\effcover_q(\MSL\otimes\QQ)^+$ by the universal property of the $q$-th effective cover, inducing the commutative triangle:
    \[\begin{tikzcd}
    A_q \arrow[r] \arrow[rd, "\varphi_{F,q}^+"'] & \effcover_q(\MSL\otimes\QQ)^+ \arrow[d] \\
    & (\MSL\otimes\QQ)^+.  
    \end{tikzcd}\]
    We claim that the top morphism is a section of the map $\effcover_q(\MSL\otimes\QQ)^+\to \slice_q(\MSL\otimes\QQ)^+\cong A_q$. 
    Since the composition 
    \[ A_q\to\effcover_q(\MSL\otimes\QQ)^+\to A_q \]
    is an endomorphism of $\Sigma^{q+(q)}\EM V$ for a $\QQ$-vector space $V$, it suffices to verify the claim after applying the $q$-th slice functor.
    However, then all above morphisms become identities (up to the identification of all terms with $A_q$ via canonical isomorphisms and $\slice_q(\MSL\otimes\QQ)^+\cong A_q$). 
    This proves the claim.

    Applying the claim by induction, for every $q$ there exists a morphism $\theta_{q+1}$ such that the map 
    $$ (\varphi_{F,0}^+\ \dots\ \varphi_{F,q}^+\ \theta_{q+1})\colon A_0\oplus \dots \oplus A_q\oplus \effcover_{q+1}(\MSL\otimes\QQ)^+\to (\MSL\otimes\QQ)^+ $$
    is an equivalence of rational motivic spectra over $F$. 
    Taking the colimit over $q$ and using the vanishing
    $$\colim_{q\to \infty} \effcover_{q+1}(\MSL\otimes\QQ)^+=0,$$ 
    which follows from the convergence of the slice filtration for $(\MSL\otimes\QQ)^+$, the result follows.
\end{proof}
\begin{remark}
    The proof of Proposition \ref{prop:Q_plus_part} implies that, over a field $F$, the slice spectral sequence for $\MSL\otimes\QQ$ degenerates. 
    Alternatively, this can be seen as follows. 
    Consider the canonical morphism $\MSL\otimes\QQ\to \MGL\otimes\QQ$. 
    Applying the $q$-th slice functor, the resulting map 
    \begin{align}\label{eq:s_q(MSL+toMGL)}
    \slice_q(\MSL\otimes\QQ)\to \slice_q(\MGL\otimes\QQ)
    \end{align}
    can be identified with
    $$ \Sigma^{q+(q)}\EM\QQ\otimes \cyc_{2q}(\W)\to \Sigma^{q+(q)}\EM\QQ\otimes \MU_{2q}, $$
    which is induced by the inclusion $\cyc_{2q}(\W)\hookrightarrow \MU_{2q}$. 
    After tensoring with $\QQ$, this inclusion becomes a split monomorphism and, fixing a retraction, yields a retraction to the map \eqref{eq:s_q(MSL+toMGL)}. 
    It follows that the slice spectral sequence for $\MSL\otimes\QQ$ degenerates at the $E^1$-page, since it does for $\MGL\otimes\QQ$; see \cite[Remark 2.8]{RSO19}.
\end{remark}
\begin{notation}
    Similarly to the plus part, for an arbitrary qcqs scheme $S$, there exists a morphism 
    \[ \varphi^-_S\colon\W\QQ[u_4,u_8,\dots]\to (\MSL\otimes\QQ)^-\]
    of homotopy commutative ring spectra in $\SH(S)\otimes\QQ$, where $\W\QQ$ is the rational Witt motivic cohomology spectrum, see \cite[Definition 5.18]{DFJK} and \cite[(10)]{KRO_hermitian_k_theory_char2}. 
    The existence uses the identification $(\sph\otimes\QQ)^-\cong \W\QQ$ \cite[Corollary 6.2]{DFJK} (see also \cite[Theorem 5]{ALP} for the case of a field).
\end{notation}
\begin{lemma}\label{lemma:msl_rational_minus}
    The map $\varphi^-_S\colon\W\QQ[u_4,u_8,\dots]\to (\MSL\otimes\QQ)^-$ is an isomorphism.
\end{lemma}
\begin{proof}
    It suffices to consider the case $S=\Spec(\Z)$.
    The real realization of $\varphi^-_\Z$ is the classical decomposition of the rational oriented cobordism spectrum $\mathbf{MSO}\otimes\QQ$, since the real realization of $\Witt\QQ$ is $H\QQ$ (see \cite[Proposition H]{real_realization}) and the generators are compatible with the topological ones by construction. 
    The result follows from \cite{Bachmann_real_etale} since $\mathrm{Sper}(\Z)=\mathrm{Sper}(\RR)=*$.
\end{proof}
\begin{theorem}\label{thm:rational_decomp}
    Let $S$ be a qcqs scheme. 
    Then $\varphi_S=(\varphi^+_S,\varphi^-_S)$ is an isomorphism of homotopy commutative ring spectra in $\SH(S)$:
    \[ \varphi_S\colon\EM\QQ[y_2,y_3,\dots]\times \W\QQ[u_4,u_8,\dots]\xrightarrow{\simeq} \MSL\otimes\QQ. \]
\end{theorem}
\begin{proof}
    Combine Proposition \ref{prop:Q_plus_part} and Lemma \ref{lemma:msl_rational_minus}. 
\end{proof}
We do not know how to lift the above equivalence to a morphism of $\mathbb{E}_\infty$-ring spectra. 
However, an immediate consequence is a decomposition of the cohomology theory represented by $\MSL\otimes\QQ$.
\begin{corollary}
    Let $S$ be a regular noetherian scheme. 
    Then there is an isomorphism
    $$ \MSL^{*,*}(S,\QQ)\cong \Bigl(H^{*,*}(S,\QQ)[y_2,y_3,\dots]\Bigr)\times \Bigl(\bigl(H^{*}_{\mathrm{Zar}}(S_\QQ,\W)\otimes\QQ\bigr)[u_4,u_8,\dots]\Bigr)$$
    of bigraded rings, where $S_\QQ=S\times_{\Spec(\Z)}\Spec(\QQ)$ is the characteristic zero fiber, the generator $y_i$ has bidegree $(-2i,-i)$, and the generator $u_{4i}$ has bidegree $(-8j,-4j)$.
\end{corollary}
\begin{proof}
    According to \cite[Corollary 5.14]{DFJK}, the spectrum $\W\QQ$ represents $\QQ$-linear Zariski Witt-sheaf cohomology groups of the characteristic zero fiber over a regular noetherian scheme. 
    In turn, over regular noetherian schemes, $\EM\QQ$ represents rational motivic cohomology as defined by Beilinson. 
    The result follows from Theorem \ref{thm:rational_decomp}.
\end{proof}
\begin{remark}
    Another interesting question that we are not discussing here is the homotopy type of the $p$-local special linear algebraic cobordism spectrum $\MSL_{(p)}=\MSL\otimes{\Z_{(p)}}$ for a prime $p$. 
    The corresponding answers in topology are as follows. 
    For an odd prime $p$, both $\MSU_{(p)}$ and $\mathbf{MSO}_{(p)}$ split as wedges of suspensions of the Brown--Peterson spectrum $\mathbf{BP}$, as mentioned in \cite{Pengelley}.
    However, at the prime $2$ the special unitary cobordism $\MSU_{(2)}$ splits as a wedge of suspensions of the Brown--Peterson spectrum $\mathbf{BP}$ and Pengelley's $\mathbf{BoP}$, which might be viewed as an ``orthogonal Brown--Peterson spectrum'' \cite{Pengelley}. In turn, the spectrum $\mathbf{MSO}_{(2)}$ splits as a wedge of suspensions of the Eilenberg--MacLane spectra $H\Z_{(2)}$ and $H\Z/2$.
\end{remark}

\appendix
\section{Theorem B in characteristic 2}\label{appendix_char2}
In this appendix, we prove the main result over fields of characteristic $2$. 
Since the argument is a bit special in this case (and much easier), we decided to present it separately for the convenience of the reader.
For simplicity, the standing assumption that $F$ is a field of characteristic $2$ and $\Lambda$ is the ring $\Z[\nicefrac{1}{2}]$ will not be repeated. 

\begin{lemma}\label{lemma:d1-integralsummands_char2}
    The group  $[\EM\Lambda,\Sigma^{2+(1)}\EM\Lambda]$ is trivial.
\end{lemma}
\begin{proof}
    The proof of Lemma \ref{lemma:d1-integralsummands} applies, using $\slice_1\sph_\Lambda\cong\Sigma^{(1)}\EM\Lambda/2=0$, since $\Lambda/2=0$.
\end{proof}
From the form of the slices (see Remark \ref{remark:first_slices}) and Lemma \ref{lemma:d1-integralsummands_char2}, we obtain that all first slice differentials are trivial. 
The potentially non-trivial higher slice differentials that may affect $\pi_n\MSL_\Lambda$ for $n\leq 3$ are:
\begin{align*}
    \dd_2\colon E^2_{-n+3,0,-n}\to E^2_{-n+2,2,-n},\\
    \dd_2\colon E^2_{-n+4,0,-n}\to E^2_{-n+3,2,-n},\\
    \dd_3\colon E^3_{-n+4,0,-n}\to E^3_{-n+3,3,-n}.
\end{align*}
Note that the differential $\dd_2\colon E^2_{-n+4,1,-n}\to E^2_{-n+3,3,-n}$, which appears in Lemma \ref{lem:d2pi4s1sphere} in characteristic $\neq 2$, is automatically trivial since $\slice_1\MSL_\Lambda=0$.
\begin{lemma}\label{lemma:differentials_char2}
    All of the higher slice differentials listed above are zero.
\end{lemma}
\begin{proof}
    Since the unit map $\sph_\Lambda\to \MSL_\Lambda$ induces an equivalence on the zero slice, it suffices to show that the corresponding slice differentials are trivial for the motivic sphere spectrum. 
    The $\dd_3$-differential for $\sph_\Lambda$ is trivial since its target is zero. 
    Consider the canonical map $\sph_\Lambda\to\sph_\Lambda/3$ to the Moore spectrum mod $3$. It induces injective maps on the targets of the $\dd_2$-differentials. 
    Hence, it remains to prove the vanishing of the corresponding $\dd_2$-differentials for $\sph_\Lambda/3$. 
    They have the form
    \begin{align*} & \dd_2\colon h^{n-3,n}_3\to h^{n+1,n+2}_3\oplus h^{n+2,n+2}_3, \\
    & \dd_2\colon h^{n-4,n}_3\to h^{n,n+2}_3\oplus h^{n+1,n+2}_3.
    \end{align*}
    Assume first that $F$ contains a primitive $3$rd root of unity. 
    Then the Milnor(--Witt) modules $h_3^{\star-3,\star}$ and $h_3^{\star-4,\star}$ are generated in degrees $3$ and $4$, respectively, and these generators map trivially since they exist over the subfield $\mathbb{F}_4\subset F$ of $3$-cohomological dimension one. 
    If $F$ does not contain a primitive $3$rd root of unity, the degree two field extension $F\subset F(\zeta_3)$ induces injections on the groups occurring in the differentials. 
    This concludes the proof.
\end{proof}
\begin{corollary}\label{cor:stab_char2}
    If $F$ is a field of characteristic $2$, the equality 
    \[E^1_{-n+m,q,-n}(\MSL_\Lambda)=E^\infty_{-n+m,q,-n}(\MSL_\Lambda)\]
    holds for $m\leq 3$. 
    The morphism $\MSL_\Lambda\to \kq_\Lambda$ induces isomorphisms 
    \[ E^\infty_{-n+m,q,-n}(\MSL_\Lambda)\xrightarrow{\simeq} E^\infty_{-n+m,q,-n}(\kq_\Lambda)\ \text{for}\ m\leq 3\ \text{and}\ (m,q)\neq(3,3). \]
    For $(m,q)=(3,3)$, the map $E^\infty_{-n+3,3,-n}(\MSL_\Lambda)\to E^\infty_{-n+3,3,-n}(\kq_\Lambda)$ coincides with $H^{n+3,n+3}\to 0$.
\end{corollary}
\begin{proof}
    The first claim follows from Lemma \ref{lemma:d1-integralsummands_char2} and Lemma \ref{lemma:differentials_char2}. The comparison with $\kq_\Lambda$ follows from Proposition \ref{prop:slices-msl-kq}.
\end{proof}
\begin{lemma}\label{lemma:msl_msletacomp_char2}
    The natural map $\MSL\to\MSL^\wedge_\eta$ induces isomorphisms
    $$ \pi_n\MSL_\Lambda\xrightarrow{\simeq} \pi_{n}(\MSL_\Lambda)^\wedge_\eta $$
    for $n\in\Z$. The same holds for $\kq_\Lambda$.
\end{lemma}
\begin{proof}
    More generally, any $\Lambda$-local motivic spectrum $\EE\in \SH(F)\otimes\Lambda$ splits as $\EE\cong \EE^\wedge_\eta\times \EE[\eta^{-1}]$ (recall that $\Lambda=\Z[\nicefrac{1}{2}]$ and see \cite[Remark 4]{ALP}), while the spectrum $\EE[\eta^{-1}]$ is trivial by the main result of \cite{Bachmann_real_etale} since $\mathrm{Sper}(F)=\emptyset$ for a field $F$ of positive characteristic.
    The reference applies, because inverting $\rho$ coincides with inverting $\eta$ whenever $2$ is invertible.
\end{proof}
\begin{proof}[Proof of Theorem B in characteristic $2$]
    This follows directly from Lemma \ref{lemma:msl_msletacomp_char2}, Lemma \ref{lemma:convergence}, and Corollary \ref{cor:stab_char2}.
\end{proof}

\bibliographystyle{amsalpha}
\bibliography{slices_msl}

\end{document}